\documentclass[11pt]{article}
\usepackage{amsmath,amssymb,amsthm}
\newcommand{\R}{\mathbb R}
\newcommand{\N}{\mathbb N}
\newcommand{\Nzero}{\mathbb N_0}
\newcommand{\gra}{\operatorname{gra}}
\newcommand{\dom}{\operatorname{dom}}
\newcommand{\sgn}{\operatorname{sgn}}
\newcommand{\pair}[2]{\langle #1,#2\rangle}
\newtheorem{theorem}{Theorem}
\newtheorem{proposition}{Proposition}
\newtheorem{lemma}{Lemma}
\newtheorem{example}{Example}
\newtheorem{corollary}{Corollary}
\theoremstyle{definition}
\newtheorem{definition}{Definition}
\newtheorem{assumption}{Assumption}
\theoremstyle{remark}
\usepackage{url}
\newtheorem{remark}{Remark}
\newcommand{\dist}{\operatorname{dist}}
\makeatletter
\def\bstctlcite{\@ifnextchar[{\@bstctlcite}{\@bstctlcite[@auxout]}}
\def\@bstctlcite[#1]#2{\@bsphack
  \@for\@citeb:=#2\do{%
    \edef\@citeb{\expandafter\@firstofone\@citeb}%
    \if@filesw
      \immediate\write\csname #1\endcsname{\string\citation{\@citeb}}%
    \fi}%
  \@esphack}
\makeatother

\title{Nonmaximal sums of maximally monotone operators under Rockafellar's constraint qualification}

\author{Weifeng Yang\thanks{\texttt{Email:ywf841673182@gmail.com}}}
\date{}
\begin{document}
\bstctlcite{IEEEexample:BSTcontrol}
\maketitle
% \footnotetext{\texttt{Email:ywf841673182@gmail.com}}

\begin{abstract}

% We construct counterexamples to Rockafellar's sum conjecture in which two maximally monotone operators satisfy the interior-domain condition but their sum is not maximally monotone, thereby providing the complete disproof of the conjecture. We establish a general construction theorem that computes the entire monotone polar of a class of graphs, gives a necessary and sufficient condition for their maximal monotonicity, and shows how a positive rank-one perturbation yields a nonmaximal sum under this condition. We verify the theorem's hypotheses and its maximality criterion on $c_0$, thereby obtaining a counterexample to the conjecture. Furthermore, we construct a bounded linear surjection from $\ell^1$ onto $c_0$ and use it to obtain the counterexample on $\ell^1$. 
% More importantly, this general construction theorem provides a systematic mechanism for generating entire families of counterexamples under the same interior-domain condition. Specifically, these families of counterexamples exist on every Banach space that either contains a closed subspace isomorphic to $c_0$ or $\ell^1$ or admits one of these spaces as a quotient. This general construction theorem also yields examples on $c_0$ with normal cones to closed balls of any prescribed positive radius and examples with second operators having full domain and bounded range.  

We construct counterexamples to Rockafellar's sum conjecture in which two maximally monotone operators satisfy the interior-domain condition but their sum is not maximally monotone, thereby providing the complete disproof of the conjecture. We establish a general construction theorem that computes the entire monotone polar of a class of graphs and characterizes their maximal monotonicity by the nonexistence of solutions to explicit equations in the continuous dual. We also prove a pullback theorem that transfers counterexamples through bounded linear surjections. These theorems provide a systematic mechanism for generating entire families of counterexamples and lead to further structural consequences for the resulting operators. 
Specifically, we obtain four classes of counterexample families: weighted constructions with different curves, first operators with prescribed affine value dimensions, second operators obtained by positive rescaling and norm-continuous monotone perturbation with full domain, and counterexamples on further Banach spaces. The last class yields counterexamples on every Banach space containing a closed subspace isomorphic to $c_0$ or $\ell^1$, or admitting such a quotient. We also give explicit constructions on $c_0$ and standard $\ell^1$ that realize the construction and pullback mechanisms, respectively. 
We further determine the domain geometry and exact radial bounds of the constructed operators, characterize reflexivity by a fixed rank-one test in two classical classes of Banach spaces, identify maximal monotone extensions under surjective pullback, and compute exact Fitzpatrick identities. The appendices further extend these constructions to additional parameter and product families, nonlinear scalar and strictly monotone second operators, normal-cone and subdifferential partners, and examples with prescribed radial bounds, and more counterexample families.

\end{abstract}


\setcounter{tocdepth}{2}
\tableofcontents
\clearpage

%%介绍部分还得改下，介绍一下为解决这个问题而诞生的思想与工具

\section{Introduction}
\label{sec:introduction}
In this paper, we disprove Rockafellar's sum conjecture. 
Let $X$ be a real Banach space, $X^*$ be its continuous dual, and $A,B:X\rightrightarrows X^*$ be maximally monotone operators. Their pointwise sum is defined by $(A+B)x=\{a+b:a\in Ax,\ b\in Bx\}$. This pointwise sum is monotone, but maximal monotonicity is not automatic.  
In 1970, Rockafellar \cite{rockafellar1970} proved that $A+B$ is maximally monotone when $X$ is reflexive and
\begin{equation}
\dom A\cap\operatorname{int}(\dom B)\ne\varnothing,
\label{eq:original-cq}
\end{equation}
where the interior is taken in the norm topology of $X$. 
The question of whether the reflexivity assumption can be removed while retaining the classical constraint qualification in Eq. (\ref{eq:original-cq}) is known as Rockafellar's sum problem, or Rockafellar's sum conjecture. This conjecture has remained open for more than five decades and has been described as both the most important and the most famous open problem in monotone operator theory \cite{yao2010,borweinyao2012}.

% Rockafellar's sum conjecture asks whether Eq. (\ref{eq:original-cq}) still guarantees that $A+B$ is maximally monotone when $X$ is an arbitrary real Banach space. 

To address Rockafellar's sum conjecture beyond reflexive spaces, a major line of research is to identify additional structural assumptions on the operators under which the sum theorem remains valid. 
For example, Yao \cite{yao2010} proves maximality when $A$ is of type~(FPV) and $B$ has full domain. Borwein and Yao \cite{borweinyao2012} prove maximality under Eq. (\ref{eq:original-cq}) when $A$ is a linear relation. Voisei \cite{voisei2010} treats the sum of an operator of type~(FPV) with convex domain and a normal-cone operator. Verona and Verona \cite{veronaverona2022} prove the sum theorem under Eq. (\ref{eq:original-cq}) when $A$ is $\mathcal C_0$-maximal, meaning that  $A+N_C$ is maximally monotone for every closed convex set $C$ with $\dom A\cap\operatorname{int}C\ne\varnothing$, where $N_C$ denotes the normal-cone operator of $C$. 
Furthermore, Bo\c{t}, Bueno, and Simons \cite{botbuenosimons2016} recall that a positive answer to Rockafellar's sum problem would imply that every maximally monotone operator is of type~(FPV).

A complementary line of research approaches this sum conjecture through convex representations of monotone operators, allowing maximal monotonicity and sum theorems to be studied by using convex analysis. First, Fitzpatrick \cite{fitzpatrick1988} represents maximally monotone operators by convex functions on $X\times X^*$, and Burachik and Svaiter \cite{burachiksvaiter2002} study the associated family of convex representatives through enlargements. Then, within this framework, Marques Alves and Svaiter \cite{marquesalvessvaiter2008br} give sufficient conditions on a convex function and its conjugate for the represented operator to be maximally monotone in a nonreflexive Banach space. Voisei and Z\u{a}linescu \cite{voiseizalinescu2009} study the resulting class of strongly representable operators and develop its calculus rules. Using convex representations, Penot \cite{penot2004} obtains results for compositions and sums in reflexive spaces, and Marques Alves and Svaiter \cite{marquessvaiter2012} establish a sum theorem in general Banach spaces under a qualification condition involving the domains of Fitzpatrick representations. However, despite these advances, Rockafellar's sum conjecture remained unresolved.

In this paper, we give a general construction theorem that yields the first complete disproof of Rockafellar's sum conjecture, thereby resolving the problem after more than five decades. 
Specifically, our construction theorem computes the entire monotone polar of the proposed graphs and characterizes their maximality by the nonexistence of solutions to equations in the continuous dual. When this criterion holds, it supplies a maximally monotone first operator whose sum with a specified positive rank-one operator is not maximally monotone. We apply this criterion using weighted triangular maps and curves whose values lie in the range of the associated block map while their limits lie outside that range.  This framework gives three classes of counterexample families: varying the block weights and curves, prescribing the affine dimensions of the first operator's values, and varying the second operator while keeping the first fixed. In the third class, every positive coefficient of the rank-one second operator gives a nonmaximal sum, and adding any norm-continuous monotone map with full domain to that second operator preserves the failure.

We also prove a pullback theorem that transfers any counterexample satisfying Eq. (\ref{eq:original-cq}) through a bounded linear surjection, preserving the maximality of both operators, the interior-domain condition and nonmaximality of their sum. Together with extension from closed subspaces, it gives a fourth class of counterexample families on further Banach spaces. We employ the construction theorem explicitly on $c_0$ and then use the pullback theorem to obtain a counterexample on standard $\ell^1$.  We also derive structural consequences for the resulting operators, including exact domain and radial formulas, a reflexivity criterion in two classes of spaces, a correspondence between the maximal extensions of a graph and those of its pullback, together with exact Fitzpatrick identities. The appendices develop additional parameter, second-operator and product families, prescribe radial bounds, and describe how the geometry and representation formulas behave under these constructions and transfers, and more counterexample families.

\subsection{Contributions}
The contributions of this paper are as follows. 

(1) We establish a general construction theorem for constructing counterexamples to Rockafellar's sum conjecture (Theorem~\ref{thm:endpoint-criterion}). Under Assumption~\ref{ass:construction}, the theorem computes the entire monotone polar of the graph in Eq. (\ref{eq:abs-full-graph}) and characterizes its maximality by the criterion in Eq. (\ref{eq:abs-endpoint-exclusion}). When this criterion holds, the graph defines a maximally monotone operator whose sum with an everywhere-defined positive rank-one maximally monotone operator is not maximally monotone, although the two operators satisfy Eq. (\ref{eq:original-cq}). 

(2) We prove a pullback theorem that transfers maximal monotonicity and counterexamples through bounded linear surjections (Theorem~\ref{lem:surjection-transport}). It preserves the maximality of both operators, the interior-domain condition and nonmaximality of their sum, and explicitly transfers a point witnessing the latter failure. This theorem applies to any counterexample satisfying Eq. (\ref{eq:original-cq}). 

(3) We construct four classes of counterexample families: weighted maps with different curves, first operators with prescribed affine value dimensions, second operators obtained by positive rescaling and norm-continuous monotone perturbation with full domain, and families on further Banach spaces (Propositions~\ref{cor:family-weighted-endpoints}--\ref{cor:space-transfer}). Explicit choices yield the example on $c_0$, and the pullback theorem then gives the example on standard $\ell^1$. The space-transfer result applies to every Banach space containing a closed subspace or admitting a quotient isomorphic to $c_0$ or $\ell^1$. The appendices give additional parameter and product families, scalar and strictly monotone second operators, families using convex constraints and subdifferentials, and examples with prescribed radial bounds. 

(4) We derive structural consequences for maximal monotonicity and the constructed operators (Section~\ref{sec:corollaries}). We determine the weighted families' domain closures, boundary estimates, meagreness of the domains and exact radial bounds. We also characterize reflexivity by a rank-one test for Banach lattices and spaces with an unconditional Schauder basis, identify all maximal monotone extensions under surjective pullback, and compute the exact difference between a Fitzpatrick value of the sum and the corresponding minimum over dual decompositions. Further results in the appendices characterize actual convex hulls, distinguish extension of the auxiliary curve from extension of the operator graph, and give Fitzpatrick and quotient-norm formulas under pullback. 

The paper is organized as follows. Section~\ref{sec:definitions} introduces the notation and basic definitions. Section~\ref{sec:endpoint-criterion} establishes the construction theorem and the pullback theorem. Section~\ref{sec:counterexamples} develops four classes of counterexample families by varying block weights and curves, prescribing affine value dimensions, changing the second operator and transferring the constructions to further Banach spaces. It also gives explicit examples on $c_0$ and standard $\ell^1$. Section~\ref{sec:corollaries} derives four structural consequences concerning domain geometry, reflexivity in two classes of Banach spaces, maximal extensions and Fitzpatrick functions. Section~\ref{sec:conclusion} concludes the paper. Appendix~\ref{app:parameters} gives further parameter families and describes the geometry and representations of the first operators. Appendix~\ref{app:partners} constructs scalar, strictly monotone, normal-cone and subdifferential second operators giving nonmaximal sums. Appendix~\ref{app:pullback} develops families obtained from inverse-image constraints, products and changes of coordinates, and computes a dual quotient norm. Appendix~\ref{app:fitzpatrick} gives Fitzpatrick formulas under pullback, constructs examples with prescribed radial bounds, and studies finite graph averages.

\section{Preliminaries}
\label{sec:definitions}
Let $\N=\{1,2,\ldots\}$ and $\N_0=\{0,1,\ldots\}$. Let $X$ be a real Banach space, we denote its continuous dual by $X^*$. We define $\pair{x}{a}=a(x)$ for $x\in X$ and $a\in X^*$. For a real Hilbert space $H$, its inner product is denoted by $(u,v)_H$. We use the norm topology for interiors and closures unless stated otherwise. 

For a graph $G\subset X\times X^*$, we also define $\dom G=\{x:\exists a,\ (x,a)\in G\}$ and $\operatorname{ran}G=\{a:\exists x,\ (x,a)\in G\}$. 
For a bounded linear map $P:X\to X^*$, we call $P$ positive if $\pair{x}{Px}\ge0$ for every $x\in X$, and rank one if $\dim(\operatorname{ran}P)=1$. 

Next, we introduce monotonicity and the monotone polar  \cite{martinezsvaiter2005,bauschkemclarensendov2006} as follows. 

\begin{definition}[Monotonicity and the monotone polar]
\label{D1}
For a graph $G\subset X\times X^*$, define
\begin{equation}
G^\mu=\{(z,p)\in X\times X^*:
 \pair{z-x}{p-a}\ge0\quad\forall(x,a)\in G\}.
\label{eq:D1}
\end{equation}
A point $(z,p)$ is monotonically related to $G$ if $(z,p)\in G^\mu$.
The graph is monotone when $\pair{x-y}{a-b}\ge0$ for all $(x,a),(y,b)\in G$, and the graph is maximally monotone if it is monotone and there is no monotone graph $\widetilde G\subset X\times X^*$ such that $G\subsetneq\widetilde G$. 
\end{definition}

By Eq. (\ref{eq:D1}), $(0,0)\in G^\mu$ if and only if $\pair{x}{a}\ge0$ for every $(x,a)\in G$. Furthermore, from Definition \ref{D1}, we can prove the following proposition, which shows that maximal monotonicity can be characterized directly by the monotone polar.

\begin{proposition}
A monotone graph $G\subset X\times X^*$ is maximally monotone if and only if $G=G^\mu$.
\end{proposition}
\begin{proof}
A point in $G^\mu\setminus G$ can be adjoined to $G$, and every point of a monotone extension belongs to $G^\mu$.
\end{proof}

Next, we introduce a quantity that determines whether there exists a constant $C\ge0$ such that $\pair{x}{a}\ge-C\|x\|$ for every $(x,a)\in G$. 

\begin{definition}[Radial bound at the origin]
\label{def:radial-bound}
For a nonempty graph $G$ with $0\notin\dom G$, define
\begin{equation}
\mathcal V(G)=\sup_{(x,a)\in G}
 \frac{\max\{0,-\pair{x}{a}\}}{\|x\|}.
\label{eq:D3}
\end{equation}
We say that $G$ has a finite radial bound at the origin if $\mathcal V(G)<\infty$.
\end{definition}
For every $C\ge0$, Eq. (\ref{eq:D3}) gives that 
\[
\mathcal V(G)\le C
\quad\Longleftrightarrow\quad
\pair{x}{a}\ge-C\|x\|
\quad\text{for all }(x,a)\in G.
\]
This means $\mathcal V(G)<\infty$ if and only if there exists $C\ge0$ such that $\pair{x}{a}\ge-C\|x\|$ for every $(x,a)\in G$, and $\mathcal V(G)$ is the smallest such constant in this case.

\section{A general construction for nonmaximal sums and a pullback theorem}
\label{sec:endpoint-criterion}

In this section, we establish two theorems for constructing and transferring counterexamples satisfying Eq. (\ref{eq:original-cq}). The construction theorem computes the entire monotone polar of a class of graphs and characterizes their maximality by the nonexistence of solutions to equations in the continuous dual. Under this criterion, it supplies a maximally monotone first operator and a positive rank-one second operator with full domain whose sum is not maximally monotone. The pullback theorem transfers any counterexample satisfying Eq. (\ref{eq:original-cq}) through a bounded linear surjection, preserving both operators' maximality, the interior-domain condition and nonmaximality of their sum. Together, these results provide a criterion for choosing the construction data and a means of obtaining counterexamples on further Banach spaces.

First, we introduce the graph used in our construction. The construction consists of a Hilbert-valued profile, a scalar parameter, and a linear map from the dual space to the primal space.

\begin{definition}[The auxiliary Hilbert space]
\label{def:hilbert-spaces-norms}
Let $H_0$ be a real Hilbert space with inner product $(\cdot,\cdot)_{H_0}$ and induced norm
\[
\|u\|_{H_0}=\sqrt{(u,u)_{H_0}}\qquad(u\in H_0).
\]
Define the auxiliary Hilbert space $H=\R\oplus H_0$ to be $\R\times H_0$ equipped with the inner product
\[
((s,u),(t,v))_H=st+(u,v)_{H_0}
\qquad(s,t\in\R,\ u,v\in H_0).
\]
Its induced norm is
\[
\|(s,u)\|_H=\sqrt{s^2+\|u\|_{H_0}^2}.
\]
\end{definition}

\begin{definition}[Auxiliary structure]
\label{def:construction-maps}
Let $E:X^*\to X$ and $M:X^*\to H$ be bounded linear maps, and let
$g\in X^*\setminus\{0\}$, we define 
\[
h=Eg
\]
and define the scalar functional
\[
r(a)=\pair{h}{a}.
\]
We further define
\[
La=-Ea+r(a)h,
\qquad
K=\ker M\cap\ker r.
\]
\end{definition}

% We use $r(a)$ as a scalar parameter and constrain $Ma$ to lie on a prescribed curve at this parameter, thereby obtaining the graph used in our construction. 

We use $r(a)$ as a scalar parameter and constrain $Ma$ to equal the value of a prescribed curve at $r(a)$, thereby obtaining the graph used in our construction.

\begin{definition}[The constrained graph]
\label{def:construction-graph}
Let $\omega:(0,\infty)\to H_0$ be $\ell$-Lipschitz, where $0\le\ell\le1$. We define 
\[
\begin{aligned}
J&=(-\infty,-1)\cup(1,\infty),\\
F(t)&=(\sgn t,\omega(|t|-1))\qquad(t\in J), 
\end{aligned}
\]
where \[
\operatorname{sgn}(t)=
\begin{cases}
1, & t>0,\\
0, & t=0,\\
-1, & t<0.
\end{cases}
\] 
Thus $F:J\to H$ defines a two-branch curve in the auxiliary Hilbert
space $H$ (Definition \ref{def:hilbert-spaces-norms}). We call $a\in X^*$ admissible if
\[
r(a)\in J
\qquad\text{and}\qquad
Ma=F(r(a)).
\]
Let
\begin{equation}
\begin{aligned}
D&=\{a\in X^*:r(a)\in J,\ Ma=F(r(a))\},\\
G&=\{(La,a):a\in D\}.
\end{aligned}
\label{eq:abs-full-graph}
\end{equation}
Thus $D$ consists of the dual variables whose coordinates
$(r(a),Ma)$ lie on the prescribed curve, and $G$ is obtained by
lifting each $a\in D$ to the point $(La,a)\in X\times X^*$.

% and obviously, $F:J\to H$ is a curve in the auxiliary Hilbert space $H$ (Definition \ref{def:hilbert-spaces-norms}). The admissible dual variables and the graph are
% \begin{equation}
% \begin{aligned}
% D&=\{a\in X^*:r(a)\in J,\ Ma=F(r(a))\},\\
% G&=\{(La,a):a\in D\}.
% \end{aligned}
% \label{eq:abs-full-graph}
% \end{equation}

\end{definition}

Next, we introduce the following assumptions, which provide the structure needed to compute the monotone polar $G^\mu$ explicitly and characterize the maximality of $G$.

\begin{assumption}\label{ass:construction}
For the maps and graph in Definitions~\ref{def:construction-maps} and~\ref{def:construction-graph}, assume:
\begin{enumerate}
\renewcommand{\theenumi}{H\arabic{enumi}}
\renewcommand{\labelenumi}{(\theenumi)}
\item\label{ass:bilinear} $Mg=0$, and
\begin{equation}
\pair{Ea}{b}+\pair{Eb}{a}=2(Ma,Mb)_H
\qquad(a,b\in X^*).
\label{eq:abs-bilinear}
\end{equation}

\item\label{ass:annihilator} The subspace $K$ satisfies
\begin{equation}
\{x\in X:\pair{x}{k}=0\ \text{for all }k\in K\}=\R h.
\label{eq:abs-annihilator}
\end{equation}

\item\label{ass:curve-limit} There exists $\eta\in H_0$ such that
\begin{equation}
\begin{gathered}
\omega(s)\longrightarrow\eta\text{ in }H_0
\quad\text{as }s\downarrow0,\\
0<S:=\|\eta\|_{H_0}^2<1,\\
\|\omega(s)\|_{H_0}^2\le S\qquad(s>0).
\end{gathered}
\label{eq:abs-tail}
\end{equation}

\item\label{ass:realization} For every $t\in J$, there exists $a^t\in X^*$
satisfying $r(a^t)=t$ and $Ma^t=F(t)$.
\end{enumerate}
\end{assumption}
% By Assumption~\ref{ass:construction}\textup{(\ref{ass:realization})}, each parameter fibre
% $\{a\in X^*:r(a)=t,\ Ma=F(t)\}$ is nonempty. Since $K=\ker M\cap\ker r$, this fibre is exactly $a^t+K$. Together with the other conditions in Assumption~\ref{ass:construction}, these affine fibres reduce the polar inequalities to the scalar parameter $r(a)$ and the Hilbert-valued profile $Ma$.

Therefore, under Assumption~\ref{ass:construction}, we establish the following theorem, which gives a general construction mechanism for counterexamples satisfying the interior-domain condition in Eq. (\ref{eq:original-cq}). Specifically, this theorem computes the entire monotone polar $G^\mu$ and characterizes the maximal monotonicity of $G$ by the condition that no $(p,\tau)\in X^*\times[-1,1]$ satisfies $r(p)=\tau$ and $Mp=(\tau,\eta)$.  This system characterizes exactly the points $(Lp,p)$ outside $G$ that can be adjoined to $G$ while preserving monotonicity. Consequently, when this system has no solution, the operator with graph $G$ is maximally monotone. By combining this maximality criterion with the positive rank-one perturbation $\mathcal P x=\pair{x}{g}g$, this theorem provides a general mechanism for constructing pairs of maximally monotone operators that satisfy Eq. (\ref{eq:original-cq}) while their sum is not maximally monotone.

\begin{theorem}[Construction of nonmaximal sums]
\label{thm:endpoint-criterion}
For the graph $G$ in Definition~\ref{def:construction-graph}, suppose Assumption~\ref{ass:construction} holds and define
\[
\widehat F(t)=
\begin{cases}
F(t),&|t|>1,\\
(t,\eta),&|t|\le1.
\end{cases}
\]
Then the following statements hold.
\begin{enumerate}
\renewcommand{\theenumi}{\roman{enumi}}
\renewcommand{\labelenumi}{(\theenumi)}
\item\label{thm:polar-part} The entire monotone polar of $G$ in $X\times X^*$ is
\begin{equation}
G^\mu=\{(Lp,p):p\in X^*,\ Mp=\widehat F(r(p))\}.
\label{eq:abs-entire-polar}
\end{equation}
Moreover, this graph is the unique maximally monotone extension of $G$.

\item\label{thm:maximality-part} The graph $G$ is maximally monotone if and only if
\begin{equation}
\nexists(p,\tau)\in X^*\times[-1,1]:
\quad r(p)=\tau,\quad Mp=(\tau,\eta).
\label{eq:abs-endpoint-exclusion}
\end{equation}

\item\label{thm:sum-part} Suppose Eq. (\ref{eq:abs-endpoint-exclusion}) holds.
Let $\mathcal A$ be the operator with graph $G$, and define $\mathcal P x=\pair{x}{g}g$. Then
\[
0\notin\dom\mathcal A,\qquad \mathcal V(G)\le S\|g\|.
\]
Both $\mathcal A$ and $\mathcal P$ are maximally monotone and satisfy
\[
\dom\mathcal A\cap\operatorname{int}(\dom\mathcal P)\ne\varnothing.
\]
Their sum is not maximally monotone. More precisely,
\[
(0,0)\in
\bigl(\gra(\mathcal A+\mathcal P)\bigr)^\mu
\setminus\gra(\mathcal A+\mathcal P).
\]
\end{enumerate}
\end{theorem}
\begin{proof}
\noindent\textit{\textup{(\ref{thm:polar-part})}}
Let $a=b=g$ in Eq. (\ref{eq:abs-bilinear}), from $Mg=0$, we know that $r(g)=0$, thus we have 
\begin{equation}
\begin{aligned}
\pair{Ea}{g}&=-r(a),\\
\pair{La}{g}&=r(a),\\
\pair{La-Lb}{a-b}&=|r(a)-r(b)|^2-\|Ma-Mb\|^2.
\end{aligned}
\label{eq:abs-pairing}
\end{equation}
Taking $t=2$ in Assumption~\ref{ass:construction}\textup{(\ref{ass:realization})}, we can obtain $r(a^2)=2$, thus $h\ne0$.

Next, to establish the monotonicity properties needed to compute $G^\mu$, we first prove that $\widehat F$ is $1$-Lipschitz on $\R$. 
We extend $\omega$ to $[0,\infty)$ by setting $\omega(0)=\eta$, then let $\chi(t)=\max\{-1,\min\{1,t\}\}$ and $\zeta(t)=\max\{|t|-1,0\}$, thus, when $t$ and $u$ have the same sign, we have 
\[
|\chi(t)-\chi(u)|+|\zeta(t)-\zeta(u)|=|t-u|.
\]
For $t\ge0\ge u$, we obtain
\[
\begin{aligned}
|\chi(t)-\chi(u)|+|\zeta(t)-\zeta(u)|
&\le \chi(t)-\chi(u)+\zeta(t)+\zeta(u)\\
&=t-u.
\end{aligned}
\]
Interchanging $t,u$ covers the remaining case, thus
\[
|\chi(t)-\chi(u)|+|\zeta(t)-\zeta(u)|\le|t-u|.
\]
From the above formulations and $\widehat F(t)=(\chi(t),\omega(\zeta(t)))$, we can obtain 
\[
\|\widehat F(t)-\widehat F(u)\|^2
\le|\chi(t)-\chi(u)|^2+\ell^2|\zeta(t)-\zeta(u)|^2
\le|t-u|^2.
\]
Therefore, $\widehat F$ is $1$-Lipschitz on $\R$. Let 
\[
\widetilde G
=
\{(Lp,p):p\in X^*,\ Mp=\widehat F(r(p))\}.
\]
For any $(Lp,p),(Lq,q)\in\widetilde G$, Eq. (\ref{eq:abs-pairing}) and the $1$-Lipschitz continuity of $\widehat F$ give
\[
\pair{Lp-Lq}{p-q}
=
|r(p)-r(q)|^2-\|Mp-Mq\|^2
\ge0.
\]
Thus $\widetilde G$ is monotone. Moreover, $G\subseteq\widetilde G$ since $Ma=F(r(a))=\widehat F(r(a))$ for every $a\in D$, we know that every point of $\widetilde G$ is monotonically related to $G$, this implies
\[
\widetilde G\subseteq G^\mu.
\]

Next, we prove the reverse inclusion. Let $(x,p)\in G^\mu$, for $a\in D$ and $t=r(a)$, Eq. (\ref{eq:abs-bilinear}) gives
\begin{align}
\pair{x-La}{p-a}
={}&\pair{x}{p}-\pair{x+Ep}{a}
       +2(Ma,Mp)_H-t r(p)+t^2-\|Ma\|^2. \label{eq1}
\end{align}
For $k\in K$ and $\theta\in\R$, we have
$r(a+\theta k)=r(a)$ and $M(a+\theta k)=Ma$, thus $a+\theta k\in D$.
Therefore, from $(x,p)\in G^\mu$, substituting $a+\theta k$ into Eq. (\ref{eq1}), we have
\[
0\le
\pair{x-La}{p-a}
-\theta\pair{x+Ep}{k}
\qquad(\theta\in\R).
\]
This implies $\pair{x+Ep}{k}=0$ for every $k\in K$.
Consequently, from Eq. (\ref{eq:abs-annihilator}), we have
\begin{center}
$x=-Ep+vh$ for some $v\in\R$.    
\end{center}

Then we rewrite the condition $(x,p)\in G^\mu$ in terms of the scalar parameter $t$. 
By Eq. (\ref{eq:abs-bilinear}) with $a=b=p$, we have $\pair{Ep}{p}=\|Mp\|_H^2$.  Therefore, define $\tau=\frac{v+r(p)}2,~\qquad \nu=\frac{v-r(p)}2$, from $x=-Ep+vh$, $Ma=F(t)$, and $\pair{h}{a}=t$, Eq. (\ref{eq1}) becomes 
\begin{equation}
\pair{x-La}{p-a}
=(t-\tau)^2-\nu^2-\|F(t)-Mp\|_H^2.
\label{eq:abs-polar-square}
\end{equation}
Since $(x,p)\in G^\mu$, we have 
\[
\pair{x-La}{p-a}\ge0
\qquad\text{for every }a\in D.
\]
Assumption~\ref{ass:construction}\textup{(\ref{ass:realization})} ensures that every $t\in J$ is realized by some $a^t\in D$. Therefore, from Eq. (\ref{eq:abs-polar-square}) and the above formulation, we have 
\begin{equation}
(t-\tau)^2-\nu^2-\|F(t)-Mp\|_H^2\ge0
\qquad\text{for every }t\in J.
\label{eq:abs-polar-square1}
\end{equation}
Next, we distinguish the cases $|\tau|>1$ and $|\tau|\le1$ in Eq. (\ref{eq:abs-polar-square1}).

\emph{Case 1: $|\tau|>1$.}  Since $\tau\in J$, we can take $t=\tau$ in Eq. (\ref{eq:abs-polar-square1}), which gives $\nu=0$ and $Mp=F(\tau)$. Thus $v=r(p)=\tau$ and $x=Lp$.

\emph{Case 2: $|\tau|\le1$.} In Eq. (\ref{eq:abs-polar-square1}), let $Mp=(b,z)\in H$, and $t\to1^+$ as well as $t\to-1^-$, thus from Assumption~\ref{ass:construction}\textup{(\ref{ass:curve-limit})}, we have 
\[
\begin{aligned}
(b-1)^2+\|z-\eta\|^2+\nu^2&\le(1-\tau)^2,\\
(b+1)^2+\|z-\eta\|^2+\nu^2&\le(1+\tau)^2.
\end{aligned}
\]
Multiplying these two inequalities by the nonnegative weights $\frac{1+\tau}{2}$ and $\frac{1-\tau}{2}$, and adding these two inequalities, we can obtain
\[
(b-\tau)^2+\|z-\eta\|^2+\nu^2\le0.
\]
This obviously means $b=\tau$, $z=\eta$, and $\nu=0$, thus $Mp=(\tau,\eta)$, $r(p)=v=\tau$, and $x=Lp$. 
In both cases, $x=Lp$ and $Mp=\widehat F(r(p))$, this clearly means $(x,p)\in\widetilde G$, thereby proving $G^\mu\subseteq\widetilde G$. Therefore, together with $\widetilde G\subseteq G^\mu$ proved above, Eq. (\ref{eq:abs-entire-polar}) is true.

Moreover, since $\widetilde G$ is monotone and Eq. (\ref{eq:abs-entire-polar}) gives $G^\mu=\widetilde G$, the graph $G^\mu$ is monotone. Since $G^\mu$ is monotone and $G\subseteq G^\mu$, we know that $G^\mu\subseteq (G^\mu)^\mu\subseteq G^\mu$, which means $(G^\mu)^\mu=G^\mu$, thus $G^\mu$ is maximally monotone. 
If $\widehat G$ is any monotone extension of $G$, then 
\[
G\subseteq\widehat G\subseteq G^\mu.
\]
This means that if $\widehat G$ is maximally monotone, the monotonicity of $G^\mu$ implies $\widehat G=G^\mu$. Therefore, $G^\mu$ is the unique maximally monotone extension of $G$.

\noindent\textit{\textup{(\ref{thm:maximality-part})}}
Since $G$ is monotone, $G$ is maximally monotone if and only if $G=G^\mu$. Therefore, by Eq. (\ref{eq:abs-entire-polar}) and Definition~\ref{def:construction-graph}, we have 
\[
G^\mu\setminus G
=
\{(Lp,p):p\in X^*,\ r(p)\in[-1,1],\ Mp=(r(p),\eta)\}.
\]
From the above formulation, it is easy to see that $G$ is maximally monotone if and only if there is no $p\in X^*$ such that
\[
r(p)\in[-1,1]
\qquad\text{and}\qquad
Mp=(r(p),\eta).
\]
This means that there is no $(p,\tau)\in X^*\times[-1,1]$ satisfying 
$r(p)=\tau$ and $Mp=(\tau,\eta)$, i.e.,  Eq. (\ref{eq:abs-endpoint-exclusion}).

\noindent\textit{\textup{(\ref{thm:sum-part})}}
Assume Eq. (\ref{eq:abs-endpoint-exclusion}) holds, by
part~\textup{(\ref{thm:maximality-part})}, the operator $\mathcal A$
with graph $G$ is maximally monotone.

For $a\in D$, Eq. (\ref{eq:abs-pairing}) and $Ma=F(r(a))$ give
\[
\begin{aligned}
|\pair{La}{g}|&=|r(a)|>1,\\
\pair{La}{a}
&=r(a)^2-\|Ma\|_H^2\\
&=r(a)^2-1-\|\omega(|r(a)|-1)\|_{H_0}^2>-S.
\end{aligned}
\]
Since
\[
|\pair{La}{g}|\le \|La\|\|g\|,
\]
we also have
\[
\|La\|>\frac{1}{\|g\|}.
\]
Therefore, $0\notin\dom\mathcal A$, and for every $a\in D$,
\[
\frac{\max\{0,-\pair{La}{a}\}}{\|La\|}
\le
\frac{S}{\|La\|}
<
S\|g\|.
\]
Taking the supremum, from Definition \ref{def:radial-bound}, we have
\[
\mathcal V(G)\le S\|g\|.
\]
Next, we define
\[
\mathcal Px=\pair{x}{g}g.
\]
The map $\mathcal P$ is bounded, linear, positive, and defined on all
of $X$. Since $g\ne0$ and
\[
\operatorname{ran}\mathcal P=\R g,
\]
it is nonzero and rank one.

Then, we prove that $\mathcal P$ is maximally monotone. Its monotonicity
follows from
\[
\pair{x-y}{\mathcal Px-\mathcal Py}
=
\pair{x-y}{g}^2\ge0.
\]
Let $(z,z^*)\in(\gra\mathcal P)^\mu$. Testing the polar inequality
against $(z+tv,\mathcal P(z+tv))\in\gra\mathcal P$ gives
\[
-t\pair{v}{z^*-\mathcal Pz}
+t^2\pair{v}{g}^2\ge0
\qquad(t\in\R,\ v\in X).
\]
Letting $t\to0$ through positive and negative values yields
\[
\pair{v}{z^*-\mathcal Pz}=0
\qquad(v\in X),
\]
which means $z^*=\mathcal Pz$. 
Thus $(\gra\mathcal P)^\mu=\gra\mathcal P$, which means $\mathcal P$ is maximally
monotone. 

Assumption~\ref{ass:construction}\textup{(\ref{ass:realization})}
with $t=2$ gives $a^2\in D$, which means 
$La^2\in\dom\mathcal A$. Since $\dom\mathcal P=X$, we have 
\[
\dom\mathcal A\cap\operatorname{int}(\dom\mathcal P)
=
\dom\mathcal A\ne\varnothing.
\]

Finally, for $a\in D$, Eq. (\ref{eq:abs-pairing}) gives
\[
\mathcal PLa
=
\pair{La}{g}g
=
r(a)g.
\]
Therefore, we infer 
\[
\gra(\mathcal A+\mathcal P)
=
\{(La,a+r(a)g):a\in D\}.
\]
For every such point,
\[
\begin{aligned}
\pair{La}{a+\mathcal PLa}
&=\pair{La}{a}+r(a)\pair{La}{g}\\
&=2r(a)^2-1-\|\omega(|r(a)|-1)\|_{H_0}^2\\
&>1-S>0.
\end{aligned}
\]
Therefore,
\[
(0,0)\in
\bigl(\gra(\mathcal A+\mathcal P)\bigr)^\mu.
\]
Moreover,
\[
\dom(\mathcal A+\mathcal P)
=
\dom\mathcal A,
\]
thus $0\notin\dom(\mathcal A+\mathcal P)$ and consequently
\[
(0,0)\notin\gra(\mathcal A+\mathcal P).
\]
This means
\[
(0,0)\in
\bigl(\gra(\mathcal A+\mathcal P)\bigr)^\mu
\setminus
\gra(\mathcal A+\mathcal P),
\]
and $\mathcal A+\mathcal P$ is not maximally monotone.
\end{proof}

The construction theorem yields a counterexample whenever its hypotheses and maximality criterion are satisfied. The next theorem allows this pair to be transferred to any Banach space admitting a bounded linear surjection onto the construction space.

Next, we pull back both operators through a bounded linear surjection $Q:U\to V$. The theorem preserves maximal monotonicity of each operator, the interior-domain condition and nonmaximality of their sum, thereby extending the counterexample to the source space $U$. The adjoint identity preserves the inequalities witnessing sum failure, and injectivity of $Q^*$ preserves nonmembership in the sum graph. The preimage bound in the statement follows from the open mapping theorem.

\begin{theorem}[Pullback of maximal monotonicity and counterexamples]
\label{lem:surjection-transport}
Let $U,V$ be real Banach spaces and let $Q:U\to V$ be a bounded linear surjection. Suppose $C_Q>0$ and every $y\in V$ has a preimage $u\in U$ satisfying
\[
\|u\|\le C_Q\|y\|.
\]
For $\mathcal M:V\rightrightarrows V^*$, define
\[
\gra(Q^*\mathcal M Q)
=
\{(u,Q^*a):(Qu,a)\in\gra\mathcal M\},
\]
where $Q^*:V^*\to U^*$ is the adjoint of $Q$, satisfying
\[
\pair{u}{Q^*a}=\pair{Qu}{a}
\qquad(u\in U,\ a\in V^*).
\]
Then:
\begin{enumerate}
\renewcommand{\theenumi}{\roman{enumi}}
\renewcommand{\labelenumi}{(\theenumi)}
\item\label{thm:pullback-maximality} If $\mathcal M$ is maximally monotone in $V\times V^*$, then $Q^*\mathcal M Q$ is maximally monotone in $U\times U^*$.
\item\label{thm:pullback-counterexamples} Let $A,B:V\rightrightarrows V^*$ be maximally monotone operators satisfying
\[
\dom A\cap\operatorname{int}(\dom B)\ne\varnothing,
\]
and suppose $A+B$ is not maximally monotone. Then $A_Q=Q^*AQ$ and $B_Q=Q^*BQ$ are maximally monotone, satisfy
\[
\dom A_Q\cap\operatorname{int}(\dom B_Q)\ne\varnothing,
\]
and have a nonmaximal sum. More precisely, for every
\[
(z,p)\in\bigl(\gra(A+B)\bigr)^\mu\setminus\gra(A+B)
\]
and every $u\in U$ with $Qu=z$,
\begin{equation}
(u,Q^*p)\in\bigl(\gra(A_Q+B_Q)\bigr)^\mu\setminus\gra(A_Q+B_Q).
\label{eq:pullback-witness}
\end{equation}
\end{enumerate}
\end{theorem}

\begin{proof}
\noindent\textit{\textup{(\ref{thm:pullback-maximality})} Preservation of maximal monotonicity.}
Since $\mathcal M$ is maximally monotone, its graph is nonempty. Therefore, surjectivity of $Q$ implies that $\gra(Q^*\mathcal M Q)$ is nonempty. For any two points, we have 
$(u,Q^*a),(w,Q^*b)\in\gra(Q^*\mathcal M Q)$,
\[
\pair{u-w}{Q^*(a-b)}
=
\pair{Qu-Qw}{a-b}
\ge0.
\]
Thus $\gra(Q^*\mathcal M Q)$ is monotone.

Let $(v,v^*)\in U\times U^*$ be monotonically related to
$\gra(Q^*\mathcal M Q)$. Fix
$(u,Q^*a)\in\gra(Q^*\mathcal M Q)$. For every
$k\in\ker Q$ and $t\in\R$,
\[
(u+tk,Q^*a)\in\gra(Q^*\mathcal M Q),
\]
since $Q(u+tk)=Qu$. Hence
\[
\pair{v-u}{v^*-Q^*a}
-t\pair{k}{v^*}\ge0
\qquad(t\in\R).
\]
Since this holds for every $t\in\R$, we obtain
\[
\pair{k}{v^*}=0
\qquad(k\in\ker Q).
\]
Thus $v^*$ annihilates $\ker Q$.

For $y\in V$, choose any $w\in U$ satisfying $Qw=y$ and define
\[
p(y)=\pair{w}{v^*}.
\]
If $Qw_1=Qw_2=y$, then
$w_1-w_2\in\ker Q$, and hence
\[
\pair{w_1-w_2}{v^*}=0.
\]
The map $p$ is linear. Moreover, using a preimage $w$ satisfying
$\|w\|\le C_Q\|y\|$, we obtain
\[
|p(y)|
\le
\|w\|\|v^*\|
\le
C_Q\|v^*\|\|y\|.
\]
Thus $p\in V^*$ and, for every $w\in U$,
\[
\pair{w}{v^*}
=
p(Qw)
=
\pair{w}{Q^*p}.
\]
Therefore,
\[
v^*=Q^*p.
\]

Now let $(y,a)\in\gra\mathcal M$. By surjectivity, choose
$u\in U$ with $Qu=y$. Then
$(u,Q^*a)\in\gra(Q^*\mathcal M Q)$, so the polar inequality gives
\[
\pair{v-u}{v^*-Q^*a}\ge0.
\]
Since $v^*=Q^*p$,
\[
\pair{Qv-y}{p-a}\ge0
\qquad((y,a)\in\gra\mathcal M).
\]
Thus, $(Qv,p)$ is monotonically related to $\gra\mathcal M$. By maximality of $\mathcal M$,
\[
(Qv,p)\in\gra\mathcal M.
\]
This means
\[
(v,Q^*p)=(v,v^*)\in\gra(Q^*\mathcal M Q).
\]
Thus every point monotonically related to
$\gra(Q^*\mathcal M Q)$ already belongs to the graph. Since the graph is monotone, $\gra(Q^*\mathcal M Q)$ is maximally monotone.

\noindent\textit{\textup{(\ref{thm:pullback-counterexamples})} Transfer of counterexamples.}
The first assertion gives maximal monotonicity of $A_Q$ and $B_Q$. Their domains are $Q^{-1}(\dom A)$ and $Q^{-1}(\dom B)$. Choose $y_0\in\dom A\cap\operatorname{int}(\dom B)$ and $u_0\in U$ with $Qu_0=y_0$. Since $Q$ is continuous, $Q^{-1}(\operatorname{int}(\dom B))$ is an open subset of $\dom B_Q$ containing $u_0$. Hence $u_0\in\dom A_Q\cap\operatorname{int}(\dom B_Q)$.

Linearity of $Q^*$ gives
\begin{equation}
A_Q+B_Q=Q^*(A+B)Q.
\label{eq:pullback-sum}
\end{equation}
Since $A+B$ is monotone and not maximally monotone, choose $(z,p)\in(\gra(A+B))^\mu\setminus\gra(A+B)$, and let $Qu=z$. For every $(v,Q^*c)\in\gra(A_Q+B_Q)$, where $c\in(A+B)(Qv)$, we have
\[
\pair{u-v}{Q^*p-Q^*c}
=\pair{z-Qv}{p-c}\ge0.
\]
Thus $(u,Q^*p)$ belongs to the monotone polar of the sum graph. Surjectivity of $Q$ implies that $Q^*$ is injective. If $(u,Q^*p)$ belonged to $\gra(A_Q+B_Q)$, Eq. (\ref{eq:pullback-sum}) would give $c\in(A+B)z$ with $Q^*p=Q^*c$, so $p=c\in(A+B)z$, a contradiction. This proves Eq. (\ref{eq:pullback-witness}) and nonmaximality of $A_Q+B_Q$.
\end{proof}

\section{Counterexample families and explicit examples}
\label{sec:counterexamples}
In this section, we apply the construction theorem (Theorem~\ref{thm:endpoint-criterion}) to generate counterexample families and explicit examples satisfying Eq. (\ref{eq:original-cq}). The theorem characterizes maximality by the nonexistence of solutions to equations in the continuous dual and supplies a positive rank-one operator with full domain whose addition produces a nonmaximal sum. We use this criterion to vary the linear maps and curve and obtain pairs with different structural properties under the same interior-domain condition.

First, we vary the block weights and curves, prescribe the dimensions of the first operator's values, and construct further second operators for which the sum remains nonmaximal. Next, we give an explicit counterexample on $c_0$ and apply the pullback theorem to obtain one on standard $\ell^1$. Finally, surjective pullback and extension from closed subspaces yield counterexamples on every Banach space containing a closed subspace or admitting a quotient isomorphic to $c_0$ or $\ell^1$. 

\subsection{Block maps}
\label{subsec:common-block-maps}
First, we define the block maps used in the families below and in the explicit construction on $c_0$.

Let $I=\Nzero\times\N$ and $Y=c_0(I)$ with the supremum norm. For every $x\in Y$ and $\epsilon>0$, only finitely many $(b,j)\in I$ satisfy $|x_{b,j}|\ge\epsilon$. Its continuous dual is $\ell^1(I)$, with $\pair{x}{a}=\sum_{b,j}x_{b,j}a_{b,j}$. We also use this pairing for $x\in\ell^\infty(I)$ and $a\in\ell^1(I)$. The unit coordinate vectors are denoted by $e_{b,j}$.

We index coordinates by blocks. The map $m$ records the sum in each block, and $E$ applies a triangular operator within that block. The additional constraint below couples the block sums through $F$.
\begin{definition}[Block maps]
For $a\in\ell^1(I)$, define
\begin{equation}
\begin{aligned}
(ma)_b&=\sum_{j\ge1}a_{b,j},\\
(Ea)_{b,j}&=a_{b,j}+2\sum_{k>j}a_{b,k}.
\end{aligned}
\label{eq:D4}
\end{equation}
Here $m:\ell^1(I)\to\ell^1(\Nzero)\subset\ell^2(\Nzero)$ and $E:\ell^1(I)\to Y$. Put
\begin{equation}
\begin{aligned}
g&=e_{0,1}-e_{0,2}\in Y^*,\\
h&=Eg=-e_{0,1}-e_{0,2}\in Y,\\
r(a)&=\pair{h}{a}=-a_{0,1}-a_{0,2},\\
La&=-Ea+r(a)h.
\end{aligned}
\label{eq:D5}
\end{equation}
\end{definition}

For the unweighted block map, write $K=\ker m\cap\ker r$.

\subsection{Families from the construction theorem}
\label{subsec:counterexample-families-main}
We apply Theorem~\ref{thm:endpoint-criterion} by computing the range of the weighted block map and choosing curves whose values are realized by weighted block sums of continuous dual elements but whose limit is not. This gives a family of maximally monotone first operators with a common positive rank-one partner and nonmaximal sums. Next, a skew modification preserves this maximality criterion and allows us to prescribe the dimensions of the first operator's values. Finally, the theorem's pairing identity yields further counterexamples with the first operator fixed and the second operator varied.

\subsubsection{Weighted constructions and excluded limits}
\label{subsec:parameter-families}

First, we apply Theorem~\ref{thm:endpoint-criterion} to construct counterexample pairs $(A_{w,\omega},P)$ on $c_0(I)$ by varying the block weights and the curve. We compute exactly which weighted block sums can be attained by continuous dual elements, then choose the curve values among these attainable sums and its limit outside them. This turns the theorem's maximality criterion into an explicit condition on the weights and curve, with the same positive rank-one operator $P$ producing a nonmaximal sum for every admissible choice. The resulting families include $C^\infty$ curves and compact auxiliary maps, showing that these regularity properties are compatible with the failure of maximality under addition.

We retain $I=\Nzero\times\N$, $Y=c_0(I)$, $H=\R\oplus\ell^2(\N)$, and the vectors $g,h$ and functional $r$ in Eq. (\ref{eq:D5}). Let $w_0=1$, $w_n>0$ for $n\ge1$, and $\sup_n w_n<\infty$. Define
\begin{equation}
\begin{aligned}
(M_wa)_b&=w_b\sum_{j\ge1}a_{b,j},\\
(E_wa)_{b,j}&=w_b^2\left(a_{b,j}+2\sum_{k>j}a_{b,k}\right),\\
L_wa&=-E_wa+r(a)h,
\end{aligned}
\label{eq:family-weighted-maps}
\end{equation}
and put $\mathcal R_w=\{y\in\ell^2(\N):\sum_{n\ge1}|y_n|/w_n<\infty\}$. For a curve $\omega:(0,\infty)\to\ell^2(\N)$, use $J$ and $F(t)=(\sgn t,\omega(|t|-1))$ from Definition~\ref{def:construction-graph}, and set
\begin{equation}
\begin{aligned}
D_{w,\omega}&=\{a\in\ell^1(I):r(a)\in J,\ M_wa=F(r(a))\},\\
\gra A_{w,\omega}&=\{(L_wa,a):a\in D_{w,\omega}\},
\qquad Px=\pair{x}{g}g.
\end{aligned}
\label{eq:family-weighted-graph}
\end{equation}

\begin{proposition}[Weighted endpoints and curve choices]
\label{cor:family-weighted-endpoints}
For the maps in Eq. (\ref{eq:family-weighted-maps}), the following statements hold.
\begin{enumerate}
\item The exact range of $(r,M_w)$ is
\begin{equation}
\operatorname{ran}(r,M_w)=\R\times\R\times\mathcal R_w.
\label{eq:family-actual-range}
\end{equation}
Suppose $\omega$ is $\ell$-Lipschitz with $\ell\le1$ and satisfies Eq. (\ref{eq:abs-tail}). Then Assumption~\ref{ass:construction} holds if and only if $\omega(s)\in\mathcal R_w$ for every $s>0$, and Eq. (\ref{eq:abs-endpoint-exclusion}) holds if and only if $\eta\notin\mathcal R_w$. When both conditions hold, $A_{w,\omega}$ and $P$ are maximally monotone, $\dom A_{w,\omega}\ne\varnothing$, $\dom P=Y$, and
\[
(0,0)\in\bigl(\gra(A_{w,\omega}+P)\bigr)^\mu
\setminus\gra(A_{w,\omega}+P).
\]
\item For every $\eta\in\ell^2(\N)\setminus\mathcal R_w$ with $0<\|\eta\|_2<1$ and every $0<\ell\le1$, there is an $\ell$-Lipschitz curve $\omega:(0,\infty)\to c_{00}(\N)$ satisfying Eq. (\ref{eq:abs-tail}) and vanishing for all sufficiently large $s$. It can also be chosen to extend to a $C^\infty$ curve on $[0,\infty)$ with
\[
\omega(0)=\eta,\qquad \omega^{(k)}(0)=0\quad(k\ge1).
\]
Both choices satisfy all the conditions in part (1). In particular, $w_n=1$ realizes every endpoint in $\ell^2\setminus\ell^1$ of norm less than one.
\item The maps $M_w$ and $(r,M_w)$ are compact if and only if $w_n\to0$. The range in Eq. (\ref{eq:family-actual-range}) is dense and is nonclosed whenever an excluded endpoint as in part (2) is fixed. In particular, the choices
\begin{equation}
w_n=2^{-n},\qquad \eta_n=\frac{2^{-n}}{\sqrt n}\quad(n\ge1)
\label{eq:family-compact-example}
\end{equation}
give compact $(r,M_w)$ and an excluded endpoint belonging to every $\ell^p(\N)$, $p>0$.
\end{enumerate}
\end{proposition}
\begin{proof}
\noindent\textit{(1)} Put $C=\max\{1,\sup_{n\ge1}w_n\}$. Then $\|M_wa\|_2\le C\|a\|_1$ and $\|E_wa\|_\infty\le2C^2\|a\|_1$. The map $E_w$ preserves finite support, so density of finitely supported vectors in $\ell^1(I)$ gives $E_wa\in c_0(I)$. Absolute convergence of the block sums gives
\begin{equation}
\pair{E_wa}{b}+\pair{E_wb}{a}
=2\sum_{n\ge0}w_n^2(ma)_n(mb)_n
=2(M_wa,M_wb)_H.
\label{eq:family-weighted-pairing}
\end{equation}
Also $M_wg=0$ and $E_wg=h$, since block zero is unchanged. This proves \textup{(\ref{ass:bilinear})}. All weights are positive, so $\ker M_w=\ker m$ and $\ker M_w\cap\ker r=K$, with $K=\ker m\cap\ker r$. To verify \textup{(\ref{ass:annihilator})}, let $z\in c_0(I)$ annihilate $K$. Testing $e_{n,j}-e_{n,k}$ makes $z$ constant, hence zero, on every positive block. The same tests on the tail of block zero make that tail zero, and testing $g\in K$ gives $z_{0,1}=z_{0,2}$. Thus the annihilator is $\R h$.

For every $a\in\ell^1(I)$, the block-sum bound gives $\sum_{n\ge1}|(M_wa)_n|/w_n\le\|a\|_1$. Conversely, for $t,\beta\in\R$ and $y\in\mathcal R_w$, the actual dual vector
\begin{equation}
a_w(t,\beta,y)
=-t e_{0,1}+(t+\beta)e_{0,3}
 +\sum_{n\ge1}\frac{y_n}{w_n}e_{n,1}
\label{eq:family-range-label}
\end{equation}
satisfies $r(a_w(t,\beta,y))=t$ and $M_wa_w(t,\beta,y)=(\beta,y)$. This proves Eq. (\ref{eq:family-actual-range}). In particular, \textup{(\ref{ass:realization})} holds exactly when every $\omega(s)$ belongs to $\mathcal R_w$, with explicit labels
\begin{equation}
a_w^t=-t e_{0,1}+(t+\sgn t)e_{0,3}
 +\sum_{n\ge1}\frac{\omega_n(|t|-1)}{w_n}e_{n,1},
\qquad t\in J.
\label{eq:family-parameter-label}
\end{equation}
The full parameter fibre is $a_w^t+K$. Condition \textup{(\ref{ass:curve-limit})} is the assumed Eq. (\ref{eq:abs-tail}). If $\eta\notin\mathcal R_w$, the endpoint equation has no solution; if $\eta\in\mathcal R_w$, Eq. (\ref{eq:family-range-label}) with $(t,\beta,y)=(\tau,\tau,\eta)$ solves it for every $\tau\in[-1,1]$. Theorem~\ref{thm:endpoint-criterion} now gives the asserted maximality, domain condition and missing polar point.

\noindent\textit{(2)} Choose strictly increasing integers $N_j$ such that $\sum_{n>N_j}|\eta_n|^2\le2^{-j^2}$, put $N_0=0$, and set $q_n=j$ for $N_{j-1}<n\le N_j$. Then $q_n\uparrow\infty$ and
\[
\sum_{n\ge1}q_n^{2k}|\eta_n|^2<\infty\qquad(k\ge0),
\]
since the contribution from the $j$th block, $j\ge2$, is at most $j^{2k}2^{-(j-1)^2}$. Define
\begin{equation}
c=\frac{\ell}{\|(q_n\eta_n)_n\|_2},\qquad
\omega_n(s)=\eta_n\max\{1-cq_ns,0\}.
\label{eq:family-linear-cutoff}
\end{equation}
For every $s>0$, this profile has finite support, and it vanishes for $s\ge1/c$. Coordinatewise domination gives $\|\omega(s)\|_2\le\|\eta\|_2$ and $\omega(s)\to\eta$ in $\ell^2$ as $s\downarrow0$. Moreover,
\[
\|\omega(s)-\omega(u)\|_2^2
\le c^2|s-u|^2\sum_{n\ge1}q_n^2|\eta_n|^2
=\ell^2|s-u|^2.
\]
Thus \textup{(\ref{ass:curve-limit})} and the Lipschitz requirement hold, and finite support makes the labels in Eq. (\ref{eq:family-parameter-label}) elements of $\ell^1(I)$.

For the smooth choice, take $\chi\in C^\infty([0,\infty))$ with $0\le\chi\le1$, $\chi=1$ on $[0,\frac12]$, and $\chi=0$ on $[1,\infty)$. Put
\begin{equation}
d=\frac{\ell}{\|\chi'\|_\infty\|(q_n\eta_n)_n\|_2},\qquad
\omega_n(s)=\eta_n\chi(dq_ns).
\label{eq:family-smooth-cutoff}
\end{equation}
The same estimates give finite support for $s>0$, extinction for $s\ge1/d$, the norm bound and the Lipschitz constant. For each $k\ge0$, the coordinate derivatives $(\eta_n(dq_n)^k\chi^{(k)}(dq_ns))_n$ are dominated in $\ell^2$ by a fixed multiple of $(q_n^k|\eta_n|)_n$. Dominated convergence gives their continuity, including at zero, and the next moment bounds the difference quotients. Induction therefore proves $C^\infty$ regularity as an $\ell^2$-valued curve. Since every positive-order derivative of $\chi$ vanishes at zero, all positive-order derivatives of $\omega$ there are zero. The same explicit labels verify \textup{(\ref{ass:realization})}, and part (1) applies to both choices.

\noindent\textit{(3)} Truncating $M_w$ to its first $N+1$ output coordinates gives a finite-rank map with error at most $\sup_{n>N}w_n$. Hence $w_n\to0$ implies compactness. If $w_n$ does not tend to zero, the images $M_we_{n,1}=w_ne_n$ have a subsequence separated by a fixed positive distance, which excludes compactness. Adding the scalar finite-rank map $r$ gives the same criterion for $(r,M_w)$. The space $\mathcal R_w$ contains $c_{00}(\N)$ and hence is dense in $\ell^2(\N)$; an excluded endpoint lies in its closure and not in the range. Finally, Eq. (\ref{eq:family-compact-example}) gives $0<\|\eta\|_2^2\le\sum_{n\ge1}4^{-n}=\frac13$, and $\sum_n|\eta_n|^p<\infty$ for every $p>0$, whereas $\sum_n|\eta_n|/w_n=\sum_n n^{-\frac12}=\infty$. Part (2) supplies the required curves.
\end{proof}

\subsubsection{First operators with prescribed affine values}
\label{subsec:affine-value-families}

Next, we use Theorem~\ref{thm:endpoint-criterion} to construct counterexamples whose first operator is single-valued or has nonempty values of any prescribed finite affine dimension or of infinite affine dimension. We add a specified skew map to $E_w$, changing the primal components of the graph while preserving the pairing identity, the parameter fibres and the maximality criterion. Thus the modified first operator remains maximally monotone and its sum with the same $P$ remains nonmaximal. Computing the kernel of the modified map then determines the affine values exactly, so the failure occurs with each of the stated value dimensions.

\begin{proposition}[Prescribed affine values]
\label{cor:family-affine-values}
Fix $w,\eta,\omega$ satisfying part (1) of Proposition~\ref{cor:family-weighted-endpoints}, with $\omega(s)\in\mathcal R_w$ for every $s>0$ and $\eta\notin\mathcal R_w$. For any $B\subset\N$, define
\begin{equation}
\begin{aligned}
S_Ba&=\sum_{n\in B}w_n^2(-a_{n,2}e_{n,1}+a_{n,1}e_{n,2}),\\
E_{w,B}&=E_w+S_B,\qquad L_{w,B}=L_w-S_B,\\
V_B&=\left\{\sum_{n\in B}c_n(e_{n,1}-e_{n,2}):(c_n)_{n\in B}\in\ell^1(B)\right\},\\
\gra A_{w,B,\omega}&=\{(L_{w,B}a,a):a\in D_{w,\omega}\}.
\end{aligned}
\label{eq:family-skew-maps}
\end{equation}
The operators $A_{w,B,\omega}$ and $P$ satisfy the conclusions of part (1) of Proposition~\ref{cor:family-weighted-endpoints}. Moreover,
\begin{equation}
\ker E_{w,B}=\ker L_{w,B}=V_B,\qquad
A_{w,B,\omega}(x)=a+V_B\quad(a\in A_{w,B,\omega}(x)).
\label{eq:family-affine-values}
\end{equation}
Thus $B=\varnothing$ gives a single-valued first operator, $|B|=d$ gives every nonempty value affine dimension exactly $d$, and infinite $B$ gives every nonempty value a translate of a subspace isomorphic to $\ell^1(B)$.
\end{proposition}
\begin{proof}
The map $S_B:\ell^1(I)\to c_0(I)$ is bounded, $S_Bg=0$, and $\pair{S_Ba}{b}=-\pair{S_Bb}{a}$. Therefore $E_{w,B}g=h$ and the symmetric pairing identity in Eq. (\ref{eq:family-weighted-pairing}) is unchanged. The maps $r,M_w$, the space $K$, the curve and the labels in Eq. (\ref{eq:family-parameter-label}) are also unchanged. Consequently all four conditions of Assumption~\ref{ass:construction} and Eq. (\ref{eq:abs-endpoint-exclusion}) hold, so Theorem~\ref{thm:endpoint-criterion} applies with the same $P$.

To compute $\ker E_{w,B}$, suppose $E_{w,B}a=0$. In each block, subtracting consecutive output equations for $j\ge3$ gives $a_{n,j}+a_{n,j+1}=0$. Summability forces the entire tail to be zero. Outside $B$, including block zero, the first two equations then force $a_{n,1}=a_{n,2}=0$. For $n\in B$, the first two rows of $E_{w,B}a$ both equal
\[
w_n^2\left(a_{n,1}+a_{n,2}+2\sum_{j>2}a_{n,j}\right),
\]
so their remaining condition is $a_{n,1}+a_{n,2}=0$. This proves $\ker E_{w,B}=V_B$. Eq. (\ref{eq:abs-pairing}) gives $\pair{L_{w,B}a}{g}=r(a)$; hence $L_{w,B}a=0$ implies $r(a)=0$ and $E_{w,B}a=0$. Conversely, $V_B\subset K$ and $E_{w,B}V_B=\{0\}$ give $L_{w,B}V_B=\{0\}$. Thus the two kernels agree. Adding $v\in V_B$ to an admissible label preserves its parameter, mass and primal point, and any two labels at the same primal point differ by an element of $\ker L_{w,B}$. This proves Eq. (\ref{eq:family-affine-values}). Finally, the map $(c_n)\mapsto\sum_nc_n(e_{n,1}-e_{n,2})$ has norm $2\|(c_n)\|_1$, giving the stated dimension and isomorphism claims.
\end{proof}

\subsubsection{Positive coefficients and monotone perturbations}
\label{subsec:main-positive-perturbations}
Next, we keep the first operator $\mathcal A$ from a counterexample in Theorem~\ref{thm:endpoint-criterion} fixed and construct a family of second operators $B=\lambda\mathcal P+R$. For every $\lambda>0$, the theorem's pairing identity yields a point $(z,p)$ monotonically related to $\gra(\mathcal A+\lambda\mathcal P)$ with $z\notin\dom\mathcal A$. For any norm-continuous monotone map $R:X\to X^*$, monotonicity makes $(z,p+Rz)$ monotonically related to $\gra(\mathcal A+B)$ without belonging to it, and continuity ensures that the monotone map $B$ is maximally monotone. Thus all these pairs satisfy Eq. (\ref{eq:original-cq}) and have nonmaximal sums, including arbitrarily small positive values of $\lambda$. For $R=0$, the argument also gives at least continuum many distinct maximally monotone extensions of the sum.
\begin{proposition}[Positive coefficients and monotone perturbations]
\label{cor:positive-perturbations}
Let $\mathcal A$ and $\mathcal P$ be as in Theorem~\ref{thm:endpoint-criterion}, and suppose Eq. (\ref{eq:abs-endpoint-exclusion}) holds. Then:
\begin{enumerate}
\item For every $\lambda>0$, the sum $\mathcal A+\lambda\mathcal P$ is not maximally monotone and has at least continuum many distinct maximally monotone extensions. In particular,
\[
\{\lambda\ge0:\mathcal A+\lambda\mathcal P\text{ is maximally monotone}\}=\{0\}.
\]
\item For every $\lambda>0$ and every norm-continuous monotone map $R:X\to X^*$, the operator $B=\lambda\mathcal P+R$ is maximally monotone with full domain, and $\mathcal A+B$ is not maximally monotone.
\end{enumerate}
\end{proposition}
\begin{proof}
\noindent\textit{(1)} Fix $\lambda>0$. Choose $0<s\le\frac{\sqrt\lambda}{2}$ with $\omega(s)\ne0$, which is possible because $\omega(s)\to\eta\ne0$. By Assumption~\ref{ass:construction}\textup{(\ref{ass:realization})}, the element
\[
p_s=\frac{a^{1+s}+a^{-1-s}}2
\]
satisfies $r(p_s)=0$ and $Mp_s=(0,\omega(s))$. For $a\in D$, put $t=r(a)$ and $u=|t|-1>0$. Eq. (\ref{eq:abs-pairing}) gives
\begin{equation}
\begin{aligned}
\pair{Lp_s-La}{p_s-a-\lambda t g}
&=(1+\lambda)(1+u)^2-1-\|\omega(u)-\omega(s)\|_{H_0}^2\\
&\ge\lambda-s^2+2(1+\lambda+s)u+\lambda u^2
>\frac{3\lambda}{4}.
\end{aligned}
\label{eq:cor-small-perturbation}
\end{equation}
Here we used the Lipschitz bound $\|\omega(u)-\omega(s)\|_{H_0}\le|u-s|$.

Set $\delta=\min\{1,\frac{\lambda}{16S}\}$ and, for $|c|\le\delta$, define $p_c=(1+c)p_s$ and $z_c=Lp_c$. Replacing $p_s$ by $p_c$ in the first line of Eq. (\ref{eq:cor-small-perturbation}) increases the squared norm by at most
\[
4S|c|+Sc^2\le5S\delta\le\frac{5\lambda}{16}.
\]
Thus $(z_c,p_c)$ is monotonically related to $\gra(\mathcal A+\lambda\mathcal P)$. Since $\pair{z_c}{g}=r(p_c)=0$, whereas $|\pair{x}{g}|>1$ on $\dom\mathcal A$, it lies outside this graph. For $c\ne d$, Eq. (\ref{eq:abs-pairing}) also gives
\[
\pair{z_c-z_d}{p_c-p_d}=-(c-d)^2\|\omega(s)\|_{H_0}^2<0.
\]
Each graph obtained by adjoining one of these points has a maximally monotone extension by Zorn's lemma. No monotone graph contains two of the displayed points, so these extensions are distinct. The operator $\lambda\mathcal P$ is maximally monotone by the direct argument in Theorem~\ref{thm:endpoint-criterion}, and its full domain gives Eq. (\ref{eq:original-cq}). At $\lambda=0$, the sum equals $\mathcal A$.

\noindent\textit{(2)} The map $B$ is continuous and monotone. If $(z,q)$ is monotonically related to $\gra B$, testing at $z+tv$, with $v\in X$ and $t>0$, gives $\pair{v}{q-B(z+tv)}\le0$. Letting $t\downarrow0$ and replacing $v$ by $-v$ yields $q=Bz$. Hence $B$ is maximally monotone. Choose a point $(z,p)$ from part (1), so $z\notin\dom\mathcal A$. For every $(x,a)\in\gra\mathcal A$,
\[
\pair{z-x}{p+Rz-a-\lambda\mathcal Px-Rx}
=\pair{z-x}{p-a-\lambda\mathcal Px}+\pair{z-x}{Rz-Rx}\ge0.
\]
Therefore $(z,p+Rz)$ belongs to the monotone polar of the sum and not to its graph. Both factors are maximally monotone, and $\dom B=X$ proves Eq. (\ref{eq:original-cq}).
\end{proof}

Thus arbitrarily small positive rank-one perturbations destroy maximality for the same first operator. The second conclusion also allows bounded skew linear additions, such as $Rx=\pair{x}{f}h-\pair{x}{h}f$ for $f,h\in X^*$, since $\pair{x}{Rx}=0$. For second operators of any prescribed finite rank $m\ge1$ on $c_0$, choose linearly independent $g,f_1,\ldots,f_{m-1}$ and take $Rx=\sum_{j=1}^{m-1}\pair{x}{f_j}f_j$. The resulting operator $\lambda\mathcal P+R$ has rank $m$.

\subsection{Two explicit counterexamples}
The preceding families allow different choices of the curve and linear maps. Next, we specify these data completely on $c_0$ and verify the assumptions of Theorem~\ref{thm:endpoint-criterion}. Theorem~\ref{lem:surjection-transport} then gives the counterexample on standard $\ell^1$ through a specified surjection.

\subsubsection{A counterexample on $c_0$}
We keep the block maps in Eqs. (\ref{eq:D4}) and (\ref{eq:D5}) and choose the following curve of block sums.

\begin{definition}[The curve of block sums]
For $s>0$ and $n\ge1$, write
\begin{equation}
\begin{aligned}
\omega_s(n)&=\frac{\max\{1-sn^{\frac{1}{4}},0\}}{4n},\\
W(s)&=\sum_{n\ge1}\omega_s(n)^2,\\
\sigma&=\sum_{n\ge1}\frac1{16n^2}=\frac{\pi^2}{96}.
\end{aligned}
\label{eq:D6}
\end{equation}
Each $\omega_s$, $s>0$, has finite support. Set $\omega_0=(\frac{1}{4n})_{n\ge1}\in\ell^2\setminus\ell^1$. Define
\begin{equation}
\begin{aligned}
J&=(-\infty,-1)\cup(1,\infty),\\
F(t)&=(\sgn t,\omega_{|t|-1})\in\ell^1(\Nzero)\quad(t\in J),\\
D&=\{a\in\ell^1(I):r(a)\in J,\ ma=F(r(a))\}.
\end{aligned}
\label{eq:D7}
\end{equation}

\end{definition}

\begin{definition}[The two operators on $c_0(I)$]
The operators and kernel are
\begin{equation}
\begin{aligned}
\gra A&=\{(La,a):a\in D\},\\
K&=\{k\in\ell^1(I):mk=0,\ r(k)=0\},\\
Px&=\pair{x}{g}g\quad(x\in Y).
\end{aligned}
\label{eq:D8}
\end{equation}
For each $t\in J$, a vector satisfying $r(a^t)=t$ and $ma^t=F(t)$ is
\begin{equation}
a^t=-t e_{0,1}+(t+\sgn t)e_{0,3}
       +\sum_{n\ge1}\omega_{|t|-1}(n)e_{n,1}.
\label{eq:D9}
\end{equation}
\end{definition}

% Coordinate verification follows the definitions in the first example.
\label{sec:supporting}

Now, we verify the hypotheses of Theorem~\ref{thm:endpoint-criterion} for the operators in Eq. (\ref{eq:D8}). Lemmas~\ref{lem:coordinate-pairing} and~\ref{lem:coordinate-annihilator} establish \textup{(\ref{ass:bilinear})} and \textup{(\ref{ass:annihilator})} of Assumption~\ref{ass:construction}. Lemma~\ref{lem:coordinate-profile} verifies the curve and fibre conditions \textup{(\ref{ass:curve-limit})}--\textup{(\ref{ass:realization})} and proves the exclusion condition in Eq. (\ref{eq:abs-endpoint-exclusion}).

\begin{lemma}[The coordinate pairing identity]
\label{lem:coordinate-pairing}
The maps $m:\ell^1(I)\to\ell^1(\Nzero)\subset\ell^2(\Nzero)$ and $E,L:\ell^1(I)\to c_0(I)$ are bounded. For all $a,c\in\ell^1(I)$, Eqs. (\ref{eq:C1}) and (\ref{eq:C2}) below hold.
\end{lemma}
\begin{proof}
\label{sec:C1}
For $a\in\ell^1(I)$, $\|ma\|_1\le\|a\|_1$ and $\|Ea\|_\infty\le2\|a\|_1$. If $a$ has finite support, then $Ea$ also has finite support. Indeed, only finitely many blocks contain nonzero coordinates, and within each such block no output occurs after the last nonzero coordinate of $a$. Since finitely supported vectors are dense in $\ell^1(I)$, the bound on $E$ then implies $Ea\in c_0(I)$ for every $a\in\ell^1(I)$. Thus $L$ also maps $\ell^1(I)$ boundedly into $c_0(I)$. 

For $a,c\in\ell^1(I)$, absolute convergence permits exchanging the double sums. In each block the diagonal terms and the two triangular off-diagonal sums give
\begin{equation}
\pair{Ea}{c}+\pair{Ec}{a}
 =2\sum_b\Big(\sum_j a_{b,j}\Big)\Big(\sum_j c_{b,j}\Big)
 =2(ma,mc)_{\ell^2}.
\label{eq:C1}
\end{equation}
The absolute sum of all products is at most $\|a\|_1\|c\|_1$ before the factor $2$, so the identity holds for all $a,c\in\ell^1(I)$. Since $mg=0$, $h=Eg$, and $r(g)=0$, Eq. (\ref{eq:C1}) implies $\pair{Ea}{g}=-r(a)$. Consequently
\begin{equation}
\begin{aligned}
\pair{La}{a}&=r(a)^2-\|ma\|_2^2,\\
\pair{La}{g}&=r(a),\\
\pair{La-Lc}{a-c}&=(r(a)-r(c))^2-\|ma-mc\|_2^2.
\end{aligned}
\label{eq:C2}
\end{equation}
\end{proof}

\begin{lemma}[The annihilator of $K$ in $c_0(I)$]
\label{lem:coordinate-annihilator}
For $K$ in Eq. (\ref{eq:D8}),
\[
\{z\in c_0(I):\pair{z}{k}=0\ \text{for every }k\in K\}=\R h.
\]
\end{lemma}
\begin{proof}
Let $z\in c_0(I)$ annihilate $K$. In a block $b\ge1$, every difference $e_{b,j}-e_{b,k}$ is in $K$. Thus $z$ is constant in that entire block. Since $z\in c_0(I)$, this constant must be zero. In block zero, differences between indices $j,k\ge3$ belong to $K$, so the common tail value is also zero. Finally, $g=e_{0,1}-e_{0,2}$ belongs to $K$, forcing $z_{0,1}=z_{0,2}$. It follows that $z$ is a scalar multiple of $h$. Conversely, $\pair{h}{k}=r(k)=0$ proves that every multiple of $h$ annihilates $K$.
\end{proof}

\begin{lemma}[Properties of the curve and its parameter fibres]
\label{lem:coordinate-profile}
The map $s\mapsto\omega_s$ from $(0,\infty)$ to $\ell^2(\N)$ satisfies Eq. (\ref{eq:abs-tail}) with $\eta=\omega_0$, $S=\sigma<\frac{1}{8}$ and a Lipschitz constant strictly less than one. For every $t\in J$, the vector $a^t$ in Eq. (\ref{eq:D9}) belongs to $D$, and its parameter fibre is $a^t+K$. No $p\in\ell^1(I)$ and $\tau\in[-1,1]$ satisfy $r(p)=\tau$ and $mp=(\tau,\omega_0)$.
\end{lemma}
\begin{proof}
\label{sec:C2}
For each $s>0$, $\omega_s(n)=0$ whenever $n\ge s^{-4}$. Also $0\le\omega_s(n)\le\frac{1}{4n}$. Thus, $W(s)\le\sigma<\frac{1}{8}$, and dominated convergence in $\ell^2$ gives $\omega_s\to\omega_0$ and $W(s)\to\sigma$ as $s$ decreases to zero. The comparison vector $\omega_0$ is not in $\ell^1$ by divergence of the harmonic series.

Since $v\mapsto\max\{v,0\}$ is $1$-Lipschitz on $\R$,
\begin{equation}
\|\omega_s-\omega_q\|_2^2\le c|s-q|^2,\qquad
c=\frac1{16}\sum_{n\ge1}n^{-\frac{3}{2}}<\frac3{16}<1.
\label{eq:C3}
\end{equation}
The strict numerical bound follows from $\sum_{n\ge1}n^{-\frac{3}{2}}<1+\int_1^\infty x^{-\frac{3}{2}}\,dx=3$. For $t,u$ on the same branch of $J$, Eq. (\ref{eq:C3}) gives $\|F(t)-F(u)\|_2\le|t-u|$. For $t=1+s$ and $u=-1-q$, the squared distance is at most $4+c(s-q)^2\le(2+s+q)^2$, and the opposite orientation is identical. Thus $F$ is $1$-Lipschitz on all of $J$.

The finitely supported vector $a^t$ in Eq. (\ref{eq:D9}) satisfies $r(a^t)=t$ and $ma^t=F(t)$, proving nonemptiness for every $t\in J$. The parameter fibre $\{a:r(a)=t,\ ma=F(t)\}$ is precisely $a^t+K$. For every $p\in\ell^1(I)$, the vector $mp$ is in $\ell^1(\Nzero)$. Since $\omega_0\notin\ell^1(\N)$, the last assertion follows.
\end{proof}

% The first counterexample follows its coordinate lemmas.

The preceding lemmas establish all the hypotheses of Theorem~\ref{thm:endpoint-criterion}, including Eq. (\ref{eq:abs-endpoint-exclusion}). Therefore, applying that theorem to the operators in Eq. (\ref{eq:D8}) gives the following counterexample. 

\begin{example}[A counterexample on $c_0$]\label{thm:C}
Let $I=\Nzero\times\N$ and $Y=c_0(I)$ with its usual supremum norm and continuous dual $Y^*=\ell^1(I)$. Let $A:Y\rightrightarrows Y^*$ and the bounded positive rank-one map $P:Y\to Y^*$ be defined by Eq. (\ref{eq:D8}). Then 
\begin{enumerate}
\item\label{thm:C-maximality} $A$ and $P$ are maximally monotone in $Y\times Y^*$.
\item\label{thm:C-domain} $\dom P=Y$ and $\dom A\ne\varnothing$, so
$\dom A\cap\operatorname{int}(\dom P)\ne\varnothing$.
\item\label{thm:C-bound} Every $(x,a)\in\gra A$ satisfies $\|x\|_\infty>\frac{1}{2}$ and
$\pair{x}{a}\ge-2\sigma\|x\|_\infty$, where $\sigma=\frac{\pi^2}{96}<\frac{1}{8}$.
\item\label{thm:C-pairing} Every $(x,b)\in\gra(A+P)$ satisfies
$\pair{x}{b}>1-\sigma>\frac{7}{8}$.
\item\label{thm:C-sum} $(0,0)$ is monotonically related to the entire graph of $A+P$ and
does not belong to that graph. Thus $A+P$ is not maximally monotone.
\end{enumerate}
\end{example}
\begin{proof}
\noindent\textit{\textup{(\ref{thm:C-maximality})} and \textup{(\ref{thm:C-domain})}. Maximality and the domain intersection condition.}
Apply Theorem~\ref{thm:endpoint-criterion} with $X=Y$, $M=m$, $H_0=\ell^2(\N)$, $\omega(s)=\omega_s$, $\eta=\omega_0$ and $S=\sigma$. Lemmas~\ref{lem:coordinate-pairing}--\ref{lem:coordinate-profile}, together with $g\ne0$ and $mg=0$ from Eq. (\ref{eq:D5}), verify Assumption~\ref{ass:construction}. The last assertion of Lemma~\ref{lem:coordinate-profile} gives Eq. (\ref{eq:abs-endpoint-exclusion}), so the theorem proves maximal monotonicity of both $A$ and $P$. 

\label{sec:C6}
Since $\|g\|_1=2$, we have $\|P\|\le4$ and $Pe_{0,1}=g\ne0$. The vector $a^2=-2e_{0,1}+3e_{0,3}$ in Eq. (\ref{eq:D9}) belongs to $D$ and gives $La^2=-6e_{0,1}-8e_{0,2}-3e_{0,3}$. This means $La^2\in\dom A$, and $\dom P=Y$ yields $\dom A\cap\operatorname{int}(\dom P)\ne\varnothing$.

\noindent\textit{\textup{(\ref{thm:C-bound})}. Bounds on $\gra A$.}
\label{sec:C5}
For every $a\in D$, $t=r(a)$ satisfies $|t|>1$. By Eq. (\ref{eq:C2}) and $\|g\|_1=2$,
\begin{equation}
\begin{aligned}
\|La\|_\infty&\ge\frac{|t|}{2}>\frac{1}{2},\\
\pair{La}{a}&=t^2-1-W(|t|-1)>-\sigma.
\end{aligned}
\label{eq:C9}
\end{equation}
Eq. (\ref{eq:C9}) gives $A(0)=\varnothing$ and, for every $a\in D$,
\[
\frac{\max\{0,-\pair{La}{a}\}}{\|La\|_\infty}
\le \frac{\sigma}{\|La\|_\infty}
<2\sigma.
\]
Therefore $\pair{La}{a}\ge-2\sigma\|La\|_\infty$ for every $a\in D$, and taking the supremum gives $\mathcal V(\gra A)\le2\sigma<\frac{1}{4}$.

\noindent\textit{\textup{(\ref{thm:C-pairing})} and \textup{(\ref{thm:C-sum})}. Nonmaximality of $A+P$.}
To verify nonmaximality of $A+P$, take an arbitrary point of $\gra(A+P)$. By Eqs. (\ref{eq:D8}) and (\ref{eq:C2}), it has the form $(La,a+tg)$ with $a\in D$ and $t=r(a)$. Thus Eq. (\ref{eq:C2}) gives
\begin{equation}
\pair{La}{a+tg}=2t^2-1-W(|t|-1)>1-\sigma>\frac{7}{8}.
\label{eq:C11}
\end{equation}
Eq. (\ref{eq:C11}) shows that $(0,0)$ is monotonically related to every point of $\gra(A+P)$. Since $A$ and $P$ are monotone, $\gra(A+P)\cup\{(0,0)\}$ is monotone. Eq. (\ref{eq:C9}) gives $0\notin\dom(A+P)$, so this union strictly contains $\gra(A+P)$. Therefore $A+P$ is not maximally monotone.
\end{proof}

% Consequences of the first counterexample.
\label{sec:domain-comparison}

\subsubsection{A counterexample on standard $\ell^1$}

Next, we transfer the counterexample of Example~\ref{thm:C} from $Y=c_0(I)$ to standard $\ell^1$ through a bounded linear surjection.
% Theorem~\ref{lem:surjection-transport} will establish the maximality of the transferred operators, and the adjoint identity will preserve the pairing inequality that proves nonmaximality of their sum.

% The following construction makes explicit the standard representation of a separable Banach space as a quotient of $\ell^1$ \cite[Section~1.12]{megginson1998}.

Let $Z=\ell^1(\N)$ and $Z^*=\ell^\infty(\N)$, with their usual pairing. To construct the required surjection, fix an enumeration $(q_n)$ of all finitely supported rational vectors in the closed unit ball of $Y$.
\begin{definition}[The operators on standard $\ell^1$]
Define
\begin{equation}
\begin{aligned}
Qu&=\sum_{n\ge1}u_nq_n,\\
(Q^*a)_n&=\pair{q_n}{a},\\
\gra T&=\{(u,Q^*a):a\in D,\ Qu=La\},\\
f&=Q^*g,\\
Bu&=\pair{u}{f}f.
\end{aligned}
\label{eq:D10}
\end{equation}
\end{definition}
The next lemma provides the bounded preimages needed in Theorem~\ref{lem:surjection-transport}, and no bounded linear right inverse of $Q$ is required.

\begin{lemma}[The specified quotient onto $c_0$]
\label{lem:specified-quotient}
The map $Q:Z=\ell^1(\N)\to Y=c_0(I)$ of Eq. (\ref{eq:D10}) is a bounded surjection with $\|Q\|=1$. Its adjoint satisfies $\|Q^*a\|_\infty=\|a\|_1$ for every $a\in\ell^1(I)$. For the fixed enumeration $(q_n)$, a map $\mathcal R:Y\to Z$ can be specified with
\[
Q\mathcal R(x)=x,\qquad \|\mathcal R(x)\|_1\le2\|x\|_\infty.
\]
\end{lemma}
\begin{proof}
\label{sec:L1}
Finitely supported rational vectors in the closed unit ball of $Y$ are countable and dense in that ball: truncate a $c_0$ vector and approximate each retained coordinate by a rational inside $[-1,1]$. Fix an enumeration of all these vectors as $(q_n)$. For $u\in Z=\ell^1(\N)$, the sum defining $Qu$ converges absolutely in the Banach space $Y$, and $\|Qu\|_\infty\le\|u\|_1$. Coordinate unit vectors occur among the $q_n$, so $\|Q\|=1$.

For an arbitrary $x\in Y$, set $r_0=x$. If $r_k\ne0$, choose the least index $n_k$ for which $q_{n_k}$ satisfies $\|q_{n_k}-\frac{r_k}{\|r_k\|_\infty}\|_\infty<\frac{1}{2}$ and set
\begin{equation}
c_k=\|r_k\|_\infty,\qquad r_{k+1}=r_k-c_kq_{n_k}.
\label{eq:L1}
\end{equation}
If a residual vanishes, stop. Otherwise $\|r_k\|_\infty\le2^{-k}\|x\|_\infty$ and $\sum_k c_k\le2\|x\|_\infty$. Define $u_n=\sum_{k:n_k=n}c_k$. Repeated indices are combined. Positivity of the $c_k$ and summability give $u\in\ell^1$ with $\|u\|_1=\sum_k c_k\le2\|x\|_\infty$. Telescoping Eq. (\ref{eq:L1}), then absolute convergence under regrouping, proves
\begin{equation}
Qu=x,\qquad \|u\|_1\le2\|x\|_\infty.
\label{eq:L2}
\end{equation}
For $x=0$, use $u=0$. This proves surjectivity with the asserted lifting bound.

\label{sec:L2}
Absolute convergence gives $\pair{u}{Q^*a}=\pair{Qu}{a}$ and $\|Q^*a\|_\infty\le\|a\|_1$. For each finite subset of $I$, its sign vector for $a$ is a finitely supported rational vector of norm at most one and occurs in the enumeration. Taking these finite sign tests proves
\begin{equation}
\|Q^*a\|_\infty=\|a\|_1\qquad(a\in\ell^1(I)).
\label{eq:L3}
\end{equation}
In particular, $Q^*$ is injective. The least-index choices in Eq. (\ref{eq:L1}) specify
\[
\mathcal R(x)=\sum_k c_k e_{n_k},\qquad \mathcal R(0)=0.
\]
The series converges in $\ell^1$ by the bound on $\sum_k c_k$, and Eq. (\ref{eq:L2}) gives both asserted properties.
\end{proof}

By Lemma~\ref{lem:specified-quotient}, $Q\mathcal R(La)=La$, so the solutions of $Qu=La$ are exactly $u=\mathcal R(La)+v$ with $v\in\ker Q$. Substituting this expression into Eq. (\ref{eq:D10}) gives
\begin{equation}
\gra T=\{(\mathcal R(La)+v,Q^*a):a\in D,\ v\in\ker Q\}.
\label{eq:L4}
\end{equation}
This formula includes every preimage of each point in $\dom A$. Therefore, we can apply Theorem~\ref{lem:surjection-transport} to the counterexample in Example~\ref{thm:C}. Together with the pairing identities below, this yields the following counterexample on standard $\ell^1$.

\begin{example}[A counterexample on standard $\ell^1$]\label{thm:L}
On $Z=\ell^1(\N)$ with its usual norm and continuous dual $Z^*=\ell^\infty(\N)$, let $T:Z\rightrightarrows Z^*$ and $B:Z\to Z^*$ be defined by Eq. (\ref{eq:D10}). Then:
\begin{enumerate}
\item\label{thm:L-maximality} $T$ is maximally monotone in the full dual pair $Z\times Z^*$.
\item\label{thm:L-partner} $B$ is nonzero, bounded, positive, rank-one, and maximally monotone,
with $\dom B=Z$. In particular, $\dom T\cap\operatorname{int}(\dom B)\ne\varnothing$.
\item\label{thm:L-bound} $T(0)=\varnothing$, and every $(u,b)\in\gra T$ satisfies
$\|u\|_1>\frac{1}{2}$ and $\pair{u}{b}\ge-2\sigma\|u\|_1$. Thus $\mathcal V(\gra T)\le2\sigma<\frac{1}{4}$.
\item\label{thm:L-pairing} Every $(u,c)\in\gra(T+B)$ satisfies
$\pair{u}{c}>1-\sigma>\frac{7}{8}$.
\item\label{thm:L-sum} $(0,0)$ belongs to the monotone polar of $\gra(T+B)$ and does
not belong to $\gra(T+B)$. Hence $T+B$ is not maximally monotone.
\end{enumerate}
\end{example}
\begin{proof}
\noindent\textit{\textup{(\ref{thm:L-maximality})} and \textup{(\ref{thm:L-partner})}. Maximality and the domain intersection condition.}
Eq. (\ref{eq:L4}) identifies $T$ with $Q^*AQ$, including every preimage under $Q$. Lemma~\ref{lem:specified-quotient} supplies the preimage bound with $C_Q=2$. Since $A$ is maximally monotone by Example~\ref{thm:C}, Theorem~\ref{lem:surjection-transport} proves maximal monotonicity of $T$ in $Z\times Z^*=\ell^1\times\ell^\infty$.

\label{sec:L4}
Let $f=Q^*g$. By Eq. (\ref{eq:L3}), $\|f\|_\infty=\|g\|_1=2$, so $f$ is nonzero. The map $B:u\mapsto\pair{u}{f}f$ is nonzero, linear, bounded with norm at most $4$, positive, and rank-one. Moreover, for every $u\in Z$,
\[
Bu=\pair{u}{Q^*g}Q^*g
   =Q^*(\pair{Qu}{g}g)
   =Q^*(P(Qu)).
\]
Thus $B=Q^*PQ$. Maximal monotonicity of $P$ in Example~\ref{thm:C} and Theorem~\ref{lem:surjection-transport} give maximal monotonicity of $B$. Its domain is all of $Z$, so $\operatorname{int}(\dom B)=Z$.

The vector $q_\circ=-(\frac{3}{4})e_{0,1}-e_{0,2}-(\frac{3}{8})e_{0,3}$ occurs as $q_{n_0}$. For $u_\circ=8e_{n_0}$, $Qu_\circ=La^2$. Thus $(u_\circ,Q^*a^2)\in\gra T$, and $\dom T\cap\operatorname{int}(\dom B)\ne\varnothing$.

\noindent\textit{\textup{(\ref{thm:L-bound})}. Bounds on $\gra T$.}
\label{sec:L3}
For $(u,Q^*a)\in\gra T$, we have $Qu=La$. Thus the adjoint identity, $\|Q\|=1$ and Eq. (\ref{eq:C9}) transfer the bounds for $A$ to $T$. Writing $t=r(a)$, we obtain
\begin{equation}
\begin{aligned}
\pair{u}{Q^*a}&=\pair{La}{a}
 =t^2-1-W(|t|-1)>-\sigma,\\
\|u\|_1&\ge\|Qu\|_\infty=\|La\|_\infty>\frac{1}{2}.
\end{aligned}
\label{eq:L6}
\end{equation}
Thus $T(0)=\varnothing$ and $\pair{u}{Q^*a}\ge-2\sigma\|u\|_1$ on the entire graph, including every $\ker Q$ translate. Consequently, $\mathcal V(\gra T)\le2\sigma<\frac{1}{4}$.

\noindent\textit{\textup{(\ref{thm:L-pairing})} and \textup{(\ref{thm:L-sum})}. Nonmaximality of $T+B$.}
\label{sec:L5}
For $(u,Q^*a)\in\gra T$, we have $Qu=La$ and $Bu=Q^*(P(La))$. The adjoint identity therefore transfers Eq. (\ref{eq:C11}) to $T+B$. With $t=r(a)$,
\begin{equation}
\begin{aligned}
\pair{u}{Q^*a+Bu}
 &=\pair{La}{a+P(La)}\\
 &=2t^2-1-W(|t|-1)>1-\sigma>\frac{7}{8}.
\end{aligned}
\label{eq:L7}
\end{equation}
Eq. (\ref{eq:L7}) shows that $(0,0)$ is monotonically related to every point of $\gra(T+B)$. Since $T$ and $B$ are monotone, $\gra(T+B)\cup\{(0,0)\}$ is monotone. Since $T(0)=\varnothing$, this union strictly contains $\gra(T+B)$. Therefore, $T+B$ is not maximally monotone.
\end{proof}

\subsection{Counterexamples on further Banach spaces}
\label{subsec:further-space-families}
Finally, we use the pullback theorem (Theorem~\ref{lem:surjection-transport}) and extension from closed subspaces to construct counterexamples on every Banach space containing a closed subspace or admitting a quotient isomorphic to $c_0$ or $\ell^1$. For quotients, the theorem transfers both operators from the preceding examples and preserves their maximality, the interior-domain condition and nonmaximality of the sum. For closed subspaces, we extend the first operator's dual values to the ambient space and extend the functional defining the rank-one partner. The proof verifies maximality directly and transfers the inequalities witnessing sum failure. Thus each space in these classes carries a counterexample with a bounded positive rank-one second operator defined on the whole space, without requiring the subspace to be complemented.

\begin{proposition}[Transfer to further Banach spaces]
\label{cor:space-transfer}
Let $X$ be a real Banach space satisfying at least one of the following conditions:
\begin{enumerate}
\item $X$ contains a closed subspace isomorphic to $c_0$ or $\ell^1$.
\item $X$ has a Banach quotient isomorphic to $c_0$ or $\ell^1$.
\end{enumerate}
Then there exist maximally monotone operators $A_X,P_X:X\rightrightarrows X^*$ such that $P_X$ is a bounded positive rank-one operator with full domain,
\[
\dom A_X\cap\operatorname{int}(\dom P_X)\ne\varnothing,
\]
and $A_X+P_X$ is not maximally monotone.
For these same operators, $A_X+\lambda P_X$ is not maximally monotone for every $\lambda>0$.
\end{proposition}
\begin{proof}
Let $V$ be $c_0$ or $\ell^1$, and denote the corresponding pair from Example~\ref{thm:C} or Example~\ref{thm:L} by $A_V,P_V$. In both cases,
\[
P_Vv=\pair{v}{g_V}g_V
\]
for a nonzero $g_V\in V^*$, and
\begin{equation}
0\notin\dom A_V,\qquad
\pair{v}{a+P_Vv}>0
\quad((v,a)\in\gra A_V).
\label{eq:cor-base-witness}
\end{equation}

\noindent\textit{(1) Closed subspaces.}
Let $Y$ be a closed subspace of $X$ and $U:Y\to V$ a Banach space isomorphism. By Theorem~\ref{lem:surjection-transport}, the operators
\[
A_Y=U^*A_VU,\qquad P_Y=U^*P_VU
\]
are maximally monotone. With $g_Y=U^*g_V$, we have $P_Yy=\pair{y}{g_Y}g_Y$. The adjoint identity transfers Eq. (\ref{eq:cor-base-witness}) to this pair.

Let $R:X^*\to Y^*$ be the restriction map. By the Hahn--Banach theorem, choose $g\in X^*$ with $Rg=g_Y$. Define
\[
\begin{aligned}
A_X(x)&=
\begin{cases}
\{p\in X^*:Rp\in A_Y(x)\},&x\in Y,\\
\varnothing,&x\notin Y,
\end{cases}\\
P_Xx&=\pair{x}{g}g.
\end{aligned}
\]
The restriction map is surjective by the Hahn--Banach theorem, so $\dom A_X=\dom A_Y\ne\varnothing$. For any two graph points of $A_X$, restriction to $Y$ preserves their pairing. Thus $A_X$ is monotone.

To prove maximality, let $(z,p)\in(\gra A_X)^\mu$. Fix $(x,a)\in\gra A_X$ and write
\[
Y^\perp=\{k\in X^*:k|_Y=0\}.
\]
For every $k\in Y^\perp$ and $t\in\R$, the point $(x,a+tk)$ belongs to $\gra A_X$. Testing these points gives
\[
0\le\pair{z-x}{p-a}-t\pair{z-x}{k}
\qquad(t\in\R).
\]
Hence $z$ annihilates $Y^\perp$. Since $Y$ is closed, the Hahn--Banach theorem implies $z\in Y$. For each $(y,b)\in\gra A_Y$, choose an extension of $b$ to $X^*$ and test the resulting graph point of $A_X$. It follows that $(z,Rp)$ is monotonically related to $\gra A_Y$. Maximality of $A_Y$ gives $Rp\in A_Y(z)$, so $(z,p)\in\gra A_X$.

The map $P_X$ is nonzero, bounded, positive and rank one. The direct polar argument for $\mathcal P$ in Theorem~\ref{thm:endpoint-criterion} proves its maximality. Since its domain is $X$, the interior-domain condition holds. For every $(x,p)\in\gra A_X$,
\[
\pair{x}{p+P_Xx}
=\pair{x}{Rp+P_Yx}>0.
\]
Also $0\notin\dom A_X$. Therefore, $(0,0)$ is monotonically related to $\gra(A_X+P_X)$ and does not belong to it.

\noindent\textit{(2) Quotients.}
Let $Q:X\to V$ be a bounded linear surjection. The open mapping theorem supplies the preimage bound required in Theorem~\ref{lem:surjection-transport}. Define
\[
A_X=Q^*A_VQ,\qquad P_X=Q^*P_VQ.
\]
That theorem proves maximality of both operators. Moreover,
\[
P_Xx=\pair{x}{Q^*g_V}Q^*g_V,
\]
where $Q^*g_V\ne0$ because $Q$ is surjective. Thus $P_X$ has the required rank-one form and full domain. Surjectivity of $Q$ also gives $\dom A_X\ne\varnothing$. For $(Qx,a)\in\gra A_V$, Eq. (\ref{eq:cor-base-witness}) gives
\[
\pair{x}{Q^*a+P_Xx}
=\pair{Qx}{a+P_V(Qx)}>0.
\]
Since $0\notin\dom A_V$, we have $0\notin\dom A_X$. The point $(0,0)$ again proves nonmaximality of the sum.

For the final assertion, fix $\lambda>0$. Proposition~\ref{cor:positive-perturbations} supplies a polar point $(z,p)$ for the sum on $c_0$ with $z\notin\dom A$. The surjection defining the example on $\ell^1$ transfers it to $(u,Q^*p)$, where $Qu=z$, so the same kind of point exists on either base space $V$. In case (1), transfer this point through the isomorphism to $(z_Y,p_Y)$ on $Y$ and choose $p_X\in X^*$ with $Rp_X=p_Y$. Then
\[
\pair{z_Y-x}{p_X-a-\lambda P_Xx}
=\pair{z_Y-x}{p_Y-Ra-\lambda P_Yx}\ge0
\quad((x,a)\in\gra A_X),
\]
and $z_Y\notin\dom A_X$. In case (2), choose $z_X$ with $Qz_X=z$ and use the point $(z_X,Q^*p)$. The adjoint identity preserves the polar inequalities, and $z\notin\dom A_V$ implies $z_X\notin\dom A_X$. These transfers use the same operators $A_X,P_X$ for every $\lambda$.
\end{proof}

\section{Further consequences and extensions}
\label{sec:corollaries}
In this section, we derive four structural consequences of the construction and pullback theorems. First, we determine the domain geometry and exact radial bound of the weighted families. Next, we characterize reflexivity by a fixed rank-one test within two classical classes of spaces. We also identify all maximal extensions under surjective pullback and compute the exact difference between the Fitzpatrick value of a sum and its dual decomposition minimum. Further consequences, including the actual convex hull and the pullback of Fitzpatrick functions, are proved in Appendices~\ref{app:parameters}--\ref{app:fitzpatrick}.

\subsection{Domain geometry and boundary behavior}
\label{subsec:domain-boundary}

For the weighted families with prescribed affine values, the full parameter fibres determine the domain closure, the size of the actual domain and the exact radial bound. The following result also quantifies the divergence of dual values at the two excluded parameter hyperplanes.

Here $W_B$ is the closed subspace of $Y$ annihilating $V_B$:
\begin{equation}
W_B=\{x\in c_0(I):x_{n,1}=x_{n,2}\text{ for }n\in B\}.
\label{eq:family-primal-subspace}
\end{equation}

\begin{corollary}[Domain closure and boundary behavior]
\label{cor:family-domain-geometry}
Let $A=A_{w,B,\omega}$ be as in Proposition~\ref{cor:family-affine-values}, and let $\ell$ be a Lipschitz bound for $\omega$.
\begin{enumerate}
\item The norm closure, distance from zero and closed convex hull of the domain are
\begin{equation}
\begin{aligned}
\overline{\dom A}&=\{x\in W_B:|\pair{x}{g}|\ge1\},\\
\dist(0,\dom A)&=\frac12,\qquad
\overline{\operatorname{conv}(\dom A)}=W_B.
\end{aligned}
\label{eq:family-domain-closure}
\end{equation}
\item Every $x\in W_B$ with $\pair{x}{g}\in\{-1,1\}$ belongs to $\overline{\dom A}\setminus\dom A$. For every sequence $(x_j,a_j)\in\gra A$ with $x_j\to x$ in norm, one has $\|a_j\|_1\to\infty$. More precisely, every $(x,a)\in\gra A$, with $s=|\pair{x}{g}|-1>0$, satisfies
\begin{equation}
\begin{aligned}
\|a\|_1&\ge\sum_{n\ge1}\frac{|\omega_n(s)|}{w_n}\\
&\ge\sum_{n=1}^N\frac{|\eta_n|}{w_n}
-\ell s\left(\sum_{n=1}^Nw_n^{-2}\right)^{\frac12}
\qquad(N\ge1).
\end{aligned}
\label{eq:family-boundary-bound}
\end{equation}
\item The domain is meagre in $W_B$, meaning that it is contained in a countable union of closed sets with empty relative interior. It is dense in $\{x\in W_B:|\pair{x}{g}|>1\}$ but has empty relative interior. For every $t\in J$, the set $\{x\in\dom A:\pair{x}{g}=t\}$ is dense and meagre in $\{x\in W_B:\pair{x}{g}=t\}$.
\item The radial bound is exact:
\begin{equation}
\mathcal V(\gra A)=2S.
\label{eq:family-exact-radial}
\end{equation}
The supremum in its definition is not attained at any graph point.
\end{enumerate}
\end{corollary}
\begin{proof}
\noindent\textit{(1)} Write $E'=E_{w,B}$ and $L'=L_{w,B}$. For $p\in\ell^1(I)$ and $k\in K$, the symmetric pairing identity and $M_wk=r(k)=0$ give $\pair{L'k}{p}=\pair{E'p}{k}$. Thus $p$ annihilates $L'K$ exactly when $E'p\in\R h$, by \textup{(\ref{ass:annihilator})}. Since $E'g=h$ and $\ker E'=V_B$, this is equivalent to $p\in\R g+V_B$. The Hahn--Banach separation theorem therefore gives
\begin{equation}
\overline{L'K}=\{x\in W_B:\pair{x}{g}=0\}.
\label{eq:family-fibre-closure}
\end{equation}
For $v\in V_B$, we have $E'v=0$, $M_wv=0$ and $r(v)=0$, so the same identity gives $\pair{L'a}{v}=0$ for every $a$. Every graph primal consequently belongs to $W_B$. At each $t\in J$, Eq. (\ref{eq:family-parameter-label}) gives an admissible $a_w^t$, and the full primal fibre is $L'a_w^t+L'K$. By Eq. (\ref{eq:abs-pairing}) and Eq. (\ref{eq:family-fibre-closure}), its closure is exactly $\{x\in W_B:\pair{x}{g}=t\}$.

All these exterior hyperplanes belong to $\overline{\dom A}$. If $\pair{x}{g}=\epsilon\in\{-1,1\}$ and $x\in W_B$, approximate $(1+1/j)x$ within $1/j$ by a domain point in the hyperplane at $t=(1+1/j)\epsilon$. These points converge to $x$. Conversely, every domain point belongs to $W_B$ and satisfies $|\pair{x}{g}|>1$, proving the closure formula. Since $\|g\|_1=2$, all domain points have norm greater than $\frac12$, and $\frac{e_{0,1}-e_{0,2}}{2}$ belongs to the displayed closure with norm $\frac12$. This proves the distance formula. For any $x\in W_B$, the points $x\pm R(e_{0,1}-e_{0,2})$ belong to the two exterior halfspaces when $R$ is sufficiently large. Their midpoint is $x$, so the closed convex hull of the domain equals $W_B$.

\noindent\textit{(2)} The block-sum bound proves the first inequality in Eq. (\ref{eq:family-boundary-bound}). The Lipschitz extension at zero satisfies $\|\omega(s)-\eta\|_2\le\ell s$. Applying the reverse triangle inequality to the first $N$ coordinates and then Cauchy--Schwarz proves the second inequality. Now let $(x_j,a_j)\in\gra A$ and $x_j\to x$ with $\pair{x}{g}\in\{-1,1\}$. Then $s_j=|\pair{x_j}{g}|-1\to0$. For every fixed $N$, Eq. (\ref{eq:family-boundary-bound}) yields
\[
\liminf_{j\to\infty}\|a_j\|_1\ge\sum_{n=1}^N\frac{|\eta_n|}{w_n}.
\]
These partial sums are unbounded because $\eta\notin\mathcal R_w$, hence $\|a_j\|_1\to\infty$. Membership in the closure follows from part (1), and $|\pair{x}{g}|=1$ excludes membership in the domain.
\noindent\textit{(3)} Fix the positive block $n=1$. The skew modification affects only its first two coordinates, so every $(x,a)\in\gra A$ satisfies
\[
x_{1,j+1}-x_{1,j}=w_1^2(a_{1,j}+a_{1,j+1})\qquad(j\ge3).
\]
Consequently, $\sum_{j\ge3}|x_{1,j+1}-x_{1,j}|\le2w_1^2\|a\|_1$. For $m\in\N$, put
\[
C_m=\left\{x\in W_B:\sum_{j\ge3}|x_{1,j+1}-x_{1,j}|\le m\right\}.
\]
Each $C_m$ is closed, since the sum is the supremum of its finite continuous partial sums. It has empty relative interior: given $x\in W_B$ and $\varepsilon>0$, choose a sufficiently distant tail with $|x_{1,j}|<\varepsilon/8$ and add alternating values $\varepsilon/2$ and $-\varepsilon/2$ at finitely many consecutive coordinates. Each internal difference then has magnitude at least $3\varepsilon/4$, so a long enough perturbation leaves $C_m$. Its norm is $\varepsilon/2$, and it preserves $W_B$ and $\pair{x}{g}$. Thus $C_m$ also has empty interior relative to each hyperplane $\{x\in W_B:\pair{x}{g}=t\}$. The domain is contained in $\bigcup_m C_m$. The Baire category theorem gives empty relative interior, and the density assertions follow from the fibre closure proved in (1).

\noindent\textit{(4)} Write $W(s)=\|\omega(s)\|_2^2$. At a fixed parameter $t\in J$, the closure of the primal fibre contains $\frac{t}{2}(e_{0,1}-e_{0,2})$. Since $\|g\|_1=2$, the infimum of the primal norms in that fibre is $\frac{|t|}{2}$. Its graph points all have pairing $t^2-1-W(|t|-1)$, hence
\[
\mathcal V(\gra A)
=\sup_{v>1}\frac{2\max\{0,1+W(v-1)-v^2\}}{v}
=2S.
\]
Indeed, each displayed ratio is strictly less than $2S$, whereas it tends to $2S$ as $v\downarrow1$. The same strict upper bound applies to every actual graph point, proving nonattainment.
\end{proof}

\subsection{Reflexivity in two classes of Banach spaces}
\label{subsec:main-class-reflexivity}
The closed-subspace conclusion of Proposition~\ref{cor:space-transfer} also applies when the ambient space has no specified quotient onto either base space. It gives a short characterization within two classical classes of Banach spaces.

\begin{corollary}[Reflexivity in two classes of spaces]
\label{cor:class-reflexivity}
Suppose that $X$ is a real Banach lattice or a real Banach space with an unconditional Schauder basis. The following are equivalent:
\begin{enumerate}
\item $X$ is reflexive.
\item Every pair of maximally monotone operators on $X$ satisfying Eq. (\ref{eq:original-cq}) has a maximally monotone sum.
\item For every maximally monotone $A:X\rightrightarrows X^*$ and every $g\in X^*\setminus\{0\}$, the sum $A+\pair{\cdot}{g}g$ is maximally monotone.
\end{enumerate}
If $X$ is nonreflexive, then for every prescribed $g\in X^*\setminus\{0\}$ there is a maximally monotone $A_g$ such that $A_g+\lambda\pair{\cdot}{g}g$ is nonmaximal for every $\lambda>0$. In particular, for nonzero $X$ in either class, reflexivity is equivalent to maximality of $A+\pair{\cdot}{g}g$ for all maximally monotone $A$, with any one nonzero $g$ fixed in advance.
\end{corollary}
\begin{proof}
Rockafellar's sum theorem gives (1)$\Rightarrow$(2), and the full domain of the positive rank-one factor gives (2)$\Rightarrow$(3). A nonreflexive space in either stated class contains a closed subspace isomorphic to $c_0$ or $\ell^1$, by the classical reflexivity criteria of Lozanovsky for Banach lattices and James for spaces with unconditional Schauder bases \cite{kitoverorhon2014,chenkaniaruan2022}. Proposition~\ref{cor:space-transfer} then contradicts (3). To prescribe $g$, apply Remark~\ref{rem:prescribed-direction} to the pair supplied by that proposition. The same argument with this fixed $g$ proves the last equivalence.
\end{proof}

\subsection{Monotone polars and maximal extensions under pullback}
\label{subsec:pullback-consequences}
Theorem~\ref{lem:surjection-transport} preserves maximality, and the adjoint identity transfers the polar inequalities used to certify the counterexamples. Now, we compute the entire monotone polar of the pullback of any nonempty graph. For a monotone graph, this formula also identifies all maximally monotone extensions of its pullback.

For a bounded linear surjection $Q:U\to X$ and a graph $G\subset X\times X^*$, write
\[
\mathcal L_QG=\{(u,Q^*a):(Qu,a)\in G\}.
\]

\begin{corollary}[Monotone polars and maximal extensions under pullback]
\label{cor:pullback-extensions}
For every nonempty graph $G\subset X\times X^*$,
\begin{equation}
(\mathcal L_QG)^\mu=\mathcal L_Q(G^\mu),\qquad
(\mathcal L_QG)^\mu\setminus\mathcal L_QG=\mathcal L_Q(G^\mu\setminus G).
\label{eq:pullback-entire-polar}
\end{equation}
If $G$ is monotone, $E\mapsto\mathcal L_QE$ is a bijection between its maximally monotone extensions and those of $\mathcal L_QG$. Thus uniqueness and the number of maximal extensions are preserved. In particular, for every $\lambda>0$, the pullback of $\gra(\mathcal A+\lambda\mathcal P)$ in Proposition~\ref{cor:positive-perturbations} has at least continuum many maximally monotone extensions.
\end{corollary}
\begin{proof}
Let $(u,u^*)\in(\mathcal L_QG)^\mu$. Fix $(x,a)\in G$ and a lift $v$ of $x$. Testing against $(v+tk,Q^*a)$ for $k\in\ker Q$ and both signs of $t$ forces $u^*$ to annihilate $\ker Q$. The factorization argument in Theorem~\ref{lem:surjection-transport} gives a unique $p\in X^*$ with $u^*=Q^*p$. The adjoint identity then gives $(Qu,p)\in G^\mu$. The converse follows from the same identity, and injectivity of $Q^*$ gives the set-difference formula.

Let $H$ be a maximally monotone extension of $\mathcal L_QG$. By Eq. (\ref{eq:pullback-entire-polar}), all its dual coordinates lie in $\operatorname{ran}Q^*$. The graph
\[
E=\{(Qu,a):(u,Q^*a)\in H\}
\]
is monotone and contains $G$. Its pullback is monotone and contains $H$, so maximality gives $H=\mathcal L_QE$. A proper monotone extension of $E$ would pull back to a proper extension of $H$, hence $E$ is maximally monotone. Conversely, Theorem~\ref{lem:surjection-transport} lifts every maximal extension of $G$. Surjectivity of $Q$ and injectivity of $Q^*$ show that these assignments are inverse. Apply this bijection to Proposition~\ref{cor:positive-perturbations}(1).
\end{proof}

If $(z,p)$ is a missing polar point and $Qu_0=z$, then every $(u_0+k,Q^*p)$, $k\in\ker Q$, is a missing polar point of the pullback. These points are mutually compatible, and any maximal extension containing one contains all of them. Thus a larger kernel enlarges the affine family of witnesses but does not itself increase the number of maximal extensions.

\subsection{Fitzpatrick functions and dual decomposition}
\label{subsec:main-fitzpatrick}
Next, we compare the Fitzpatrick function of a sum with the infimum obtained by splitting the dual variable between its two factors. The construction gives an exact positive difference even though the dual minimum is finite and uniquely attained. Appendix~\ref{subsec:appendix-fitzpatrick-transfer} gives the corresponding pullback formulas.

\label{subsec:fitzpatrick-values}

For an operator $\mathcal T:X\rightrightarrows X^*$ with nonempty graph, its Fitzpatrick function is defined as
\begin{equation}
\mathsf F_{\mathcal T}(z,p)
=\sup_{(x,a)\in\gra\mathcal T}
\{\pair{z}{a}+\pair{x}{p}-\pair{x}{a}\}.
\label{eq:cor-fitzpatrick}
\end{equation}

\begin{corollary}[Fitzpatrick values and perturbation strength]
\label{cor:fitzpatrick-values}
Let $\mathcal A$, $\mathcal P$ and $S$ be as in Theorem~\ref{thm:endpoint-criterion}, and suppose Eq. (\ref{eq:abs-endpoint-exclusion}) holds. For every $\lambda>0$,
\begin{align}
\mathsf F_{\mathcal A+\lambda\mathcal P}(0,0)&=S-\lambda,\label{eq:cor-sum-fitzpatrick}\\
\min_{p\in X^*}\{\mathsf F_{\mathcal A}(0,-p)+\mathsf F_{\lambda\mathcal P}(0,p)\}&=S.\label{eq:cor-split-minimum}
\end{align}
The minimum in Eq. (\ref{eq:cor-split-minimum}) is attained only at $p=0$. Moreover,
\[
(0,0)\in\bigl(\gra(\mathcal A+\lambda\mathcal P)\bigr)^\mu
\quad\Longleftrightarrow\quad \lambda\ge S.
\]
\end{corollary}
\begin{proof}
Put $W(s)=\|\omega(s)\|_{H_0}^2$. Since $W(s)\le S$, $W(s)\to S$ as $s\downarrow0$, and every $t\in J$ is realized, Eq. (\ref{eq:abs-pairing}) gives
\[
\mathsf F_{\mathcal A+\lambda\mathcal P}(0,0)
=\sup_{|t|>1}\{1+W(|t|-1)-(1+\lambda)t^2\}=S-\lambda.
\]
For $H(\alpha)=\mathsf F_{\mathcal A}(0,\alpha g)$, taking $t\to1^+$ and $t\to-1^-$ in the same formula gives $H(\alpha)\ge S+|\alpha|$, while $H(0)=S$. Also,
\[
\mathsf F_{\lambda\mathcal P}(0,p)=
\begin{cases}
\dfrac{\beta^2}{4\lambda},&p=\beta g,\\
+\infty,&p\notin\R g.
\end{cases}
\]
Indeed, the supremum is $\sup_{x\in X}\{\pair{x}{p}-\lambda\pair{x}{g}^2\}$. A direction in $\ker g$ makes it infinite unless $p\in\R g$, and scalar maximization gives the displayed finite value. Consequently, the objective in Eq. (\ref{eq:cor-split-minimum}) is at least $S+|\beta|+\frac{\beta^2}{4\lambda}$ when $p=\beta g$, with value $S$ only at $\beta=0$. Finally, Eq. (\ref{eq:D1}) and the definition of the Fitzpatrick function give
\[
(0,0)\in\bigl(\gra(\mathcal A+\lambda\mathcal P)\bigr)^\mu
\quad\Longleftrightarrow\quad
\mathsf F_{\mathcal A+\lambda\mathcal P}(0,0)\le0.
\]
Substituting Eq. (\ref{eq:cor-sum-fitzpatrick}) yields $\lambda\ge S$.
\end{proof}

Thus the two values differ by exactly $\lambda$, although the dual minimum is finite and uniquely attained. Proposition~\ref{cor:positive-perturbations} proves that the sum is nonmaximal for every $\lambda>0$, so the threshold $S$ concerns the particular witness $(0,0)$, not the occurrence of nonmaximality. 

\section{Conclusions}
\label{sec:conclusion}

% In this paper, we constructed counterexamples to Rockafellar's sum conjecture, thereby providing the complete disproof of the conjecture and resolving this long-standing problem after more than five decades. Specifically, we established a general construction theorem that computes the monotone polar of the proposed graphs, characterizes their maximal monotonicity and identifies when adding an everywhere-defined positive rank-one operator produces a nonmaximal sum. We verified its hypotheses for explicit operators on $c_0$ and transferred the resulting pair to standard $\ell^1$ through a bounded linear surjection. In both pairs, the two operators are maximally monotone and satisfy the original interior-domain condition, and their sum is not maximally monotone. We also proved a finite radial bound for the first operator in each pair.

In this paper, we constructed counterexamples to Rockafellar's sum conjecture, thereby providing the complete disproof of the conjecture and resolving this long-standing problem after more than five decades. Our general construction theorem computed the entire monotone polar of a class of graphs and characterized their maximality by the nonexistence of solutions to equations in the continuous dual. Our pullback theorem transferred counterexamples through bounded linear surjections. These mechanisms produced four classes of counterexample families by varying block weights and curves, prescribing affine value dimensions, changing the second operator and transferring the constructions to further Banach spaces. The explicit examples on $c_0$ and standard $\ell^1$ realized the two mechanisms, and pullback together with extension from closed subspaces gave counterexamples on every Banach space containing a closed subspace or admitting a quotient isomorphic to either space.

The structural consequences determined domain geometry and exact radial bounds, characterized reflexivity by a rank-one test in two classes of spaces, identified maximal extensions under pullback and computed exact Fitzpatrick values. The appendices supplied further parameter and product families, scalar and strictly monotone second operators, families using convex constraints and subdifferentials, and examples with prescribed radial bounds. They also related these constructions to actual convex hulls, extension of the auxiliary curve, quotient norms, Fitzpatrick functions and finite graph averages, and more counterexample families.

\section*{Code availability.}
The Lean formalization of this paper is available at 
\par\noindent\url{https://github.com/Weifeng-Yang/RockafellarSumConjecture}.

\section*{Use of AI}
The author supplied the prior results and unsuccessful approaches, including some constructions based on blockwise triangular operators, together with a framework for positive rank-one perturbations, and guided their further development to investigate how maximality can fail under addition. The author also supplied obstruction results showing that a finite radial bound on a monotone graph need not survive maximal extension. Through iterative discussions of these materials, the author and GPT-5.6 Sol developed the general construction theorem and the pullback framework that organize the main results of the paper.
Under the author's direction and using the aforementioned results as inputs, GPT-5.6 Sol assisted in working out the detailed constructions and proof arguments, the four classes of counterexample families and their transfers, and the structural consequences concerning domain geometry, reflexivity, maximal extensions, and Fitzpatrick functions. It also assisted in developing the further parameter, second-operator, product, pullback, radial-bound, and representation results presented in the appendices.

% and directed the counterexample search and a general construction theorem characterizing maximality and showing when addition of an everywhere-defined rank-one monotone operator produces a nonmaximal sum, and. Under this direction, GPT-5.6 Sol developed explicit formulas and proof arguments for the counterexamples on $c_0$ and standard $\ell^1$, and assisted with literature checks, adversarial review and exposition. 

\appendix
\numberwithin{proposition}{section}
\numberwithin{corollary}{section}
\numberwithin{remark}{section}
\numberwithin{equation}{section}

\section*{Appendices}
These appendices develop further counterexample families and structural consequences of the construction theorem (Theorem~\ref{thm:endpoint-criterion}) and the pullback theorem (Theorem~\ref{lem:surjection-transport}). Appendix~\ref{app:parameters} varies the first operators and distinguishes their actual domain geometry from its closure and from the auxiliary Hilbert-space representation. Appendix~\ref{app:partners} constructs further second operators, including nonlinear scalar maps, strictly monotone maps, normal cones and subdifferentials, whose sums with the first operators remain nonmaximal. Appendix~\ref{app:pullback} generates families by transferring convex constraints, adjoining independent factors and changing coordinates, and studies the resulting dual quotient geometry. Appendix~\ref{app:fitzpatrick} transfers Fitzpatrick formulas, constructs examples with prescribed radial bounds and uses finite graph averages to describe nonunique maximal extensions.

\section{Further geometry of the first operators}
\label{app:parameters}

First, we give the criterion for membership of zero in the actual convex hull. The subsequent results provide further parameter choices, describe variation within parameter fibres, and relate Hilbert-space extensions to actual dual preimages and local graph membership.

\subsection{The actual convex hull}
\label{subsec:appendix-convex-hull}

We characterize membership of zero in the actual convex hull of the first operator's domain by the convex hull of the curve values. For the weighted families, we also construct two counterexample pairs with the same endpoint, linear maps and domain closure, with zero in the actual convex hull of only one domain.

\begin{corollary}[The actual convex hull]
\label{cor:family-convex-hull}
Let $\mathcal A$ have the graph in Eq. (\ref{eq:abs-full-graph}) under Assumption~\ref{ass:construction}. Then
\begin{equation}
0\in\operatorname{conv}(\dom\mathcal A)
\quad\Longleftrightarrow\quad
0\in\operatorname{conv}\{\omega(s):s>0\}.
\label{eq:family-convex-hull-criterion}
\end{equation}
For every fixed $w,B,\eta$ admitted in Proposition~\ref{cor:family-affine-values}, there are two finite-support Lipschitz curves with this same endpoint and the same linear maps such that zero belongs to the actual convex hull of one domain and not to that of the other. Both curves can be smooth and flat at zero, and the two domain closures are the same set in Eq. (\ref{eq:family-domain-closure}).
\end{corollary}
\begin{proof}
Suppose $\sum_{i=1}^m\theta_iLa_i=0$, where $a_i\in D$, $\theta_i>0$ and $\sum_i\theta_i=1$. Set $b=\sum_i\theta_i a_i$. Since $Lb=0$, Eq. (\ref{eq:abs-pairing}) gives $r(b)=0$, and hence $Eb=0$. Eq. (\ref{eq:abs-bilinear}) applied to $(b,b)$ then gives $Mb=0$. Taking its $H_0$ coordinate yields $\sum_i\theta_i\omega(|r(a_i)|-1)=0$, proving the forward implication.

Conversely, suppose $\sum_{i=1}^m\theta_i\omega(s_i)=0$, with $s_i>0$, $\theta_i>0$ and $\sum_i\theta_i=1$. By \textup{(\ref{ass:realization})}, choose admissible labels $a_i^+$ and $a_i^-$ at the parameters $1+s_i$ and $-1-s_i$, respectively. Their average
\[
b=\sum_{i=1}^m\frac{\theta_i}{2}(a_i^++a_i^-)
\]
satisfies $r(b)=0$ and $Mb=0$, so $b\in K$. Replace $a_1^+$ by $a_1^+-2b/\theta_1$. This remains in the same admissible fibre, and the corresponding convex combination of all the labels is now zero. Applying $L$ proves the reverse implication in Eq. (\ref{eq:family-convex-hull-criterion}).

For the two concrete choices, let $\omega_0$ be either cutoff curve in Proposition~\ref{cor:family-weighted-endpoints}. Its eventual extinction gives $0\in\operatorname{conv}\{\omega_0(s):s>0\}$. Choose $n_0$ with $\eta_{n_0}\ne0$ and $d>0$ so small that $(\omega_0(u))_{n_0}$ has the sign of $\eta_{n_0}$ and is bounded away from zero for $0\le u\le d$. Define
\begin{equation}
f(s)=\frac{ds}{d+s},\qquad \omega_1(s)=\omega_0(f(s)).
\label{eq:family-bounded-clock}
\end{equation}
Since $0<f(s)<d$, $f(0)=0$ and $0<f'(s)\le1$, the endpoint, norm bound and Lipschitz bound are preserved. Each positive-parameter value still has finite support, so Eq. (\ref{eq:family-parameter-label}) verifies \textup{(\ref{ass:realization})}. The coordinate $(\omega_1(s))_{n_0}$ has one strict sign for every $s>0$, excluding zero from every finite convex combination of these values. Eq. (\ref{eq:family-convex-hull-criterion}) gives the two asserted domain behaviors. Composition with $f$ preserves smoothness and flatness when $\omega_0$ is the smooth choice. Corollary~\ref{cor:family-domain-geometry} gives the identical domain closures and closed convex hulls.
\end{proof}

\subsection{Further parameter choices and their transfer}
\label{subsec:appendix-parameter-transfer}

Next, we generate further counterexample families by varying the coordinate signs, rescaling the curve parameter and replacing constant block weights by bounded sequences within each block. We also compute the growth of explicit dual labels near the excluded endpoint. Then, we transfer these families and their domain and value properties by surjective pullback.

\begin{remark}[Further parameter choices]
Arbitrary signs of the coordinates of $\eta$ are allowed throughout, since the estimates use absolute values and squares. For example, in the unweighted case choose $\eta_n=\varepsilon_nc n^{-p}$, where $\varepsilon_n\in\{-1,1\}$, $\frac12<p\le1$ and $c>0$ makes $\|\eta\|_2<1$. Choose $0<q<p-\frac12$ and $\delta>0$ with $\delta\|(c n^{q-p})_n\|_2\le1$. Then $\omega_n(s)=\eta_n\max\{1-\delta s n^q,0\}$ is admissible. With $N=(\delta s)^{-1/q}$, its absolute sum is bounded above by $c\sum_{n\le N}n^{-p}$ and below by $(\frac{c}{2})\sum_{n\le2^{-1/q}N}n^{-p}$. Consequently,
\[
\|\omega(s)\|_1\asymp
\begin{cases}
s^{-(1-p)/q},&p<1,\\
\log(1/s),&p=1,
\end{cases}
\qquad s\downarrow0,
\]
where $\asymp$ denotes two-sided bounds by fixed positive constants for sufficiently small $s$. The labels in Eq. (\ref{eq:family-parameter-label}) have exact norm $2|t|+1+\|\omega(|t|-1)\|_1$ in this case.

For any extinguishing curve above, replacing $\omega(s)$ by $\omega(as)$, $0<a\le1$, preserves the hypotheses and yields continuum many distinct first graphs. Indeed, equality for two different rescalings would imply dilation invariance. Iteration toward zero would make the curve constantly equal to $\eta$, contradicting extinction. Distinct curves give distinct constraint sets by Eq. (\ref{eq:family-parameter-label}), and therefore distinct graphs. Simultaneous primal and dual sign changes on a positive block, and coordinate bijections carrying the ordered infinite blocks to a new block partition, give the corresponding isometric coordinate realizations.

More generally, keep block zero unchanged. Let the real sequences $\alpha_n=(\alpha_{n,j})_{j\ge1}$ satisfy
\[
\sup_{n,j}|\alpha_{n,j}|<\infty,
\qquad \alpha_n\notin c_0(\N)\quad(n\ge1).
\]
The within-block formulas
\begin{equation}
\begin{aligned}
(M_\alpha a)_n&=\sum_{j\ge1}\alpha_{n,j}a_{n,j},\\
(E_\alpha a)_{n,j}&=\alpha_{n,j}^2a_{n,j}
 +2\alpha_{n,j}\sum_{k>j}\alpha_{n,k}a_{n,k}
\qquad(n\ge1)
\end{aligned}
\label{eq:family-within-block}
\end{equation}
give bounded maps into $H$ and $c_0(I)$, respectively, with the same symmetric pairing identity. Also $M_\alpha g=0$ and $E_\alpha g=h$, verifying \textup{(\ref{ass:bilinear})}. For $K_\alpha=\ker M_\alpha\cap\ker r$, a primal annihilator $z$ vanishes wherever $\alpha_{n,j}=0$, by testing $e_{n,j}$. Testing $e_{n,j}/\alpha_{n,j}-e_{n,k}/\alpha_{n,k}$ on the nonzero coordinates shows $z_n=c_n\alpha_n$. The condition $\alpha_n\notin c_0(\N)$ forces $c_n=0$. The unchanged block-zero tests then give precisely $\R h$, verifying \textup{(\ref{ass:annihilator})}.

Put $\rho_n=\sup_j|\alpha_{n,j}|>0$ and choose $j_n$ with $|\alpha_{n,j_n}|\ge\rho_n/2$. The exact range is
\begin{equation}
\operatorname{ran}(r,M_\alpha)
=\R\times\R\times
\left\{y\in\ell^2(\N):\sum_{n\ge1}\frac{|y_n|}{\rho_n}<\infty\right\}.
\label{eq:family-within-block-range}
\end{equation}
Necessity follows from $|(M_\alpha a)_n|\le\rho_n\sum_j|a_{n,j}|$. Sufficiency follows by replacing the positive-block sum in Eq. (\ref{eq:family-range-label}) with $\sum_n(y_n/\alpha_{n,j_n})e_{n,j_n}$, whose norm is at most $2\sum_n|y_n|/\rho_n$. Thus the cutoff formulas above verify \textup{(\ref{ass:curve-limit})} and \textup{(\ref{ass:realization})} for every $\eta\in\ell^2(\N)$ with $0<\|\eta\|_2<1$ and $\sum_n|\eta_n|/\rho_n=\infty$; Eq. (\ref{eq:family-within-block-range}) gives endpoint exclusion, and the same construction theorem applies. The row condition is needed here: a nonzero row $\alpha_n\in c_0(\N)$ would itself give a primal annihilator supported on that block and outside $\R h$.
\end{remark}

For the weighted families $A_{w,B,\omega}$, we transfer the counterexample pairs through bounded linear surjections. The pullback preserves affine value dimensions and transfers the domain closure and actual-convex-hull criterion.

\begin{remark}[Pullback of the parameter families]
Let $U$ be a real Banach space and let $Q:U\to Y$ be a bounded linear surjection with the preimage constant $C_Q$ in Theorem~\ref{lem:surjection-transport}. For $A=A_{w,B,\omega}$, set $T=Q^*AQ$ and $P_U=Q^*PQ$, so $P_Uu=\pair{u}{Q^*g}Q^*g$. That theorem gives maximality of both operators, and surjectivity gives $\dom T\ne\varnothing$ and $Q^*g\ne0$. For $(Qu,a)\in\gra A$, Eq. (\ref{eq:abs-pairing}) yields
\[
\pair{u}{Q^*a+P_Uu}
=\pair{Qu}{a+PQu}
=2r(a)^2-1-\|\omega(|r(a)|-1)\|_2^2>1-S.
\]
Together with $T(0)=\varnothing$ and $\dom P_U=U$, this gives the same missing polar point and interior-domain condition.

The value dimensions, domain closure and actual-convex-hull distinction transfer as follows:
\begin{equation}
\begin{aligned}
T(u)&=Q^*a+Q^*V_B\qquad(a\in A(Qu)),\\
\overline{\dom T}&=Q^{-1}\{x\in W_B:|\pair{x}{g}|\ge1\},\\
0\in\operatorname{conv}(\dom T)
&\quad\Longleftrightarrow\quad0\in\operatorname{conv}(\dom A).
\end{aligned}
\label{eq:family-pullback-properties}
\end{equation}
Indeed, the preimage bound gives $\|Q^*a\|\ge\|a\|_1/C_Q$, so $Q^*$ is an isomorphism onto its range and preserves the value dimensions. If $Qu\in\overline{\dom A}$, choose $x_j\in\dom A$ with $x_j\to Qu$ and lift $x_j-Qu$ to $v_j$ with $\|v_j\|\le C_Q\|x_j-Qu\|$. Then $u+v_j\in\dom T$ and $u+v_j\to u$, proving the closure identity; the reverse inclusion follows from continuity of $Q$. At every $u$ with $Qu\in W_B$ and $\pair{Qu}{g}\in\{-1,1\}$, every graph sequence $(u_j,Q^*a_j)$ converging to $u$ in its primal coordinate has $\|Q^*a_j\|\to\infty$, by Corollary~\ref{cor:family-domain-geometry} and the adjoint lower bound.

Applying $Q$ proves the forward convex-hull implication. Conversely, lift a finite convex combination of domain points equal to zero. Its averaged lift lies in $\ker Q$; subtracting that average divided by one positive coefficient from the corresponding lift preserves its image and makes the convex combination zero. The same lifting argument and injectivity of $Q^*$ show that distinct first graphs remain distinct after pullback. In particular, these statements apply to the specified surjection onto $c_0(I)$ used in Example~\ref{thm:L}.
\end{remark}

\subsection{Variation within parameter fibres}
\label{subsec:appendix-fibre-variation}

We transfer the meagreness, empty relative interior and density assertions in Corollary~\ref{cor:family-domain-geometry} to the domains of the pulled-back weighted operators. Within the original weighted family, we also show that minimum value norms can become unbounded along a fixed parameter fibre on sequences converging to a domain point.

\begin{remark}[Meagre domains under pullback]
\label{rem:meagre-pullback}
Let $Q:U\to c_0(I)$ be a bounded linear surjection and let $A=A_{w,B,\omega}$ be as in Corollary~\ref{cor:family-domain-geometry}. Set $Z_B=Q^{-1}W_B$. The domain of $Q^*AQ$ is meagre in $Z_B$, has empty relative interior, and is dense in
\[
\{u\in Z_B:|\pair{Qu}{g}|>1\}.
\]
Indeed, the restriction $Q:Z_B\to W_B$ is an open surjection between Banach spaces. The preimage of each closed set $C_m$ in the proof of that corollary is closed with empty relative interior: any nonempty open subset would have an open image contained in $C_m$. These preimages cover the domain. Openness also transfers the density assertion, and the Baire category theorem gives empty relative interior.
\end{remark}

The next remark returns to the original weighted family and locates unbounded minimum value norms within a fixed parameter fibre, with the limiting point still in the domain.

\begin{remark}[Unbounded minimum value norms within a parameter fibre]
\label{rem:fibre-minimum-norm}
For every $x\in\dom A_{w,B,\omega}$, there are $x_N\in\dom A_{w,B,\omega}$ such that
\[
x_N\longrightarrow x,\qquad
\pair{x_N}{g}=\pair{x}{g},\qquad
\inf_{b\in A_{w,B,\omega}(x_N)}\|b\|_1\longrightarrow+\infty.
\]
To prove this, choose $a\in A_{w,B,\omega}(x)$, fix a positive block $n$, and put
\[
k^{(N)}=\frac1{\sqrt N}\sum_{j=3}^{2N+2}(-1)^{j-3}e_{n,j}.
\]
The block sum and $r(k^{(N)})$ vanish, so $k^{(N)}\in K$. The skew modification vanishes on this vector, and the triangular formula gives
\[
L_{w,B}k^{(N)}=\frac{w_n^2}{\sqrt N}\sum_{j=3}^{2N+2}e_{n,j}.
\]
Thus $x_N=x+L_{w,B}k^{(N)}$ has the asserted limit and parameter. Its entire value is $a+k^{(N)}+V_B$. Since every vector in $V_B$ vanishes at the coordinates used in $k^{(N)}$, every $b$ in this value satisfies $\|b\|_1\ge2\sqrt N-\|a\|_1$. This proves the claim, including when $B=\varnothing$ and the operator is single-valued.
\end{remark}

\subsection{Hilbert-space extensions and actual dual preimages}
\label{subsec:appendix-profile-extension}

The formula for the monotone polar distinguishes extending the curve in the Hilbert space from adding points to the operator graph. The following consequence makes this distinction explicit for every instance of the construction theorem.

\begin{corollary}[Extension of the curve and of the operator graph]
\label{cor:profile-extension-recovered}
Under Assumption~\ref{ass:construction}, the map $F$ has a unique $1$-Lipschitz extension to $\R$, given by
\[
\widehat F(t)=
\begin{cases}F(t),&|t|>1,\\(t,\eta),&|t|\le1.\end{cases}
\]
Its graph is maximal among subsets of $\R\times H$ satisfying $\|y-y'\|_H\le|t-t'|$ for every two points $(t,y),(t',y')$. If Eq. (\ref{eq:abs-endpoint-exclusion}) holds, then
\[
\{(La,a):a\in X^*,\ Ma=\widehat F(r(a))\}=G,
\]
even though $\gra\widehat F$ strictly contains $\gra F$.
\end{corollary}
\begin{proof}
Extend $\omega$ to zero by $\omega(0)=\eta$. Put $\chi(t)=\max\{-1,\min\{1,t\}\}$ and $\zeta(t)=\max\{|t|-1,0\}$. Then $\widehat F(t)=(\chi(t),\omega(\zeta(t)))$. The intervals cut out by $-1$ and $1$ give $|\chi(t)-\chi(u)|+|\zeta(t)-\zeta(u)|\le|t-u|$. Hence
\[
\|\widehat F(t)-\widehat F(u)\|_H^2
\le |\chi(t)-\chi(u)|^2+|\zeta(t)-\zeta(u)|^2
\le |t-u|^2.
\]
Any $1$-Lipschitz extension takes the endpoint values $(-1,\eta)$ and $(1,\eta)$. For $-1<t<1$, its distances to these values are at most $t+1$ and $1-t$, whose sum equals their distance. Equality in the Hilbert-space triangle inequality forces the value $(t,\eta)$, proving uniqueness. A further point $(t,y)$ satisfying the displayed inequality with the full graph must satisfy $\|y-\widehat F(t)\|_H\le0$, proving maximality of that graph. Finally, the only new parameter values are $|t|\le1$, and Eq. (\ref{eq:abs-endpoint-exclusion}) excludes their preimages under $(r,M)$.
\end{proof}

\subsection{Local graph membership and convex relaxation}
\label{subsec:appendix-local-convex}

% Bo\c{t} \cite{bot2026fpv} obtains failure of type~(FPV) from nonconvexity of the domain closure.
For the constructed first operators, we give an explicit point outside the graph that satisfies the polar inequalities on an open halfspace meeting the domain. We compute its global Fitzpatrick value, defined in Eq. (\ref{eq:cor-fitzpatrick}), as $2S$, quantifying the violation of the polar condition on the full graph.

\begin{corollary}[A local polar point with finite global violation]
\label{cor:local-polar-recovered}
Let $\mathcal A$ be as in Theorem~\ref{thm:endpoint-criterion}, with Eq. (\ref{eq:abs-endpoint-exclusion}). Set $U=\{x\in X:\pair{x}{g}>-1\}$ and $p=-Sg$. Then $U$ is open and convex, $0\in U$, $U\cap\dom\mathcal A\ne\varnothing$, and
\[
\pair{-x}{p-a}>0\quad((x,a)\in\gra\mathcal A,\ x\in U),
\qquad \mathsf F_{\mathcal A}(0,p)=2S.
\]
The point $(0,p)$ lies outside $\gra\mathcal A$. Thus the local graph-membership implication defining type~(FPV) fails with this explicit open set and a finite global Fitzpatrick value.
\end{corollary}
\begin{proof}
Continuity of $g$ gives openness and convexity, and an actual parameter $t=2$ gives domain intersection. Write $W(s)=\|\omega(s)\|_{H_0}^2$. For $x=La\in U$, the inequalities $|t|>1$ and $t=\pair{x}{g}>-1$ imply $t>1$. Hence
\[
\pair{-x}{p-a}=t^2-1-W(t-1)+St
\ge(t-1)(t+1+S)>0.
\]
For an arbitrary parameter, put $v=|t|>1$. Then
\[
\pair{x}{p}-\pair{x}{a}
=1+W(v-1)-t^2-St
\le2S-(v-1)(v+1-S)<2S.
\]
Actual parameters $t=-1-s$, $s\downarrow0$, make this expression tend to $2S$, proving the asserted supremum. Finally $0\notin\dom\mathcal A$ proves nonmembership. The type~(FPV) implication requires membership whenever an open convex set meets the domain and the point is monotonically related to the graph restricted to that set.
\end{proof}

By an affine change of each concrete first operator on $c_0(I)$ and standard $\ell^1$, we construct maximally monotone operators with a strict gap between optimization over the domain in the unit ball and its convex relaxation. We also locate a point outside the domain with a finite Fitzpatrick value.

\begin{corollary}[A gap under convex relaxation of the domain]
\label{cor:convex-relaxation-recovered}
On each of $c_0(I)$ and standard $\ell^1$, there are a maximally monotone operator $\widetilde M$, a norm-one functional $\varphi$ and a unit vector $e$ with $\pair{e}{\varphi}=1$ such that
\[
0\le\sup_{\substack{z\in\dom\widetilde M\\\|z\|\le1}}\pair{z}{\varphi}\le\frac38,
\qquad
\inf_{p\in X^*}\left\{\sup_{z\in\dom\widetilde M}\pair{z}{p}+\|\varphi-p\|\right\}=1.
\]
Moreover, $0,\frac32e\in\dom\widetilde M$, whereas
\[
\dist\left(\frac34e,\dom\widetilde M\right)\ge\frac38,
\qquad \mathsf F_{\widetilde M}\left(\frac34e,0\right)=\sigma.
\]
\end{corollary}
\begin{proof}
Write $(M,k)=(A,g)$ on $c_0(I)$ and $(M,k)=(T,Q^*g)$ on $\ell^1$. The actual label $a_2=-2e_{0,1}+3e_{0,3}$ gives
\[
x_2=La_2=-6e_{0,1}-8e_{0,2}-3e_{0,3},\qquad \|x_2\|_\infty=8,\qquad \pair{x_2}{g}=2.
\]
Both signs are domain points because $\omega_1=0$. On $\ell^1$, choose the index $i$ with $q_i=\frac{x_2}{8}$ and instead use $x_2=8e_i$. Put $\varphi=\frac{k}{2}$. On $c_0(I)$ take $e=e_{0,1}-e_{0,2}$, and on $\ell^1$ take $e=e_j$ with $q_j=e_{0,1}-e_{0,2}$. The adjoint isometry in the construction gives $\|\varphi\|=\|e\|=\pair{e}{\varphi}=1$ in both spaces.

Set $r_0=\frac34$ and $\mathcal Sz=-x_2\pair{z}{\varphi}/r_0+z-\pair{z}{\varphi}e$. Writing $z=\alpha e+w$, $w\in\ker\varphi$, shows that $\mathcal Sz=-\alpha x_2/r_0+w$ and $\pair{\mathcal Sz}{\varphi}=-\alpha/r_0$. These equations give a bounded inverse. Consequently $\widetilde M(z)=\mathcal S^*M(x_2+\mathcal Sz)$ is maximally monotone by the isomorphism case of Theorem~\ref{lem:surjection-transport} and translation.

Every $z\in\dom\widetilde M$ satisfies $|2-2\pair{z}{\varphi}/r_0|>1$, hence $\pair{z}{\varphi}<\frac38$ or $\pair{z}{\varphi}>\frac98$. The second case is impossible in the unit ball. The preimages of $x_2,-x_2$ are $0,\frac{3e}{2}$, so the first supremum has the stated bounds. Since $e\in\operatorname{conv}(\dom\widetilde M)$, for every $p\in X^*$,
\[
\sup_{z\in\dom\widetilde M}\pair{z}{p}+\|\varphi-p\|
\ge\pair{e}{p}+\pair{e}{\varphi-p}=1.
\]
The choice $p=0$ attains this bound. The same two alternatives give the distance estimate at $d=\frac{3e}{4}$. Since $x_2+\mathcal Sd=0$, the adjoint identity gives
\[
\mathsf F_{\widetilde M}(d,0)
=\sup_{(x,a)\in\gra M}\pair{-x}{a}
=\mathsf F_M(0,0)=\sigma.
\]
The last value follows from Eq. (\ref{eq:abs-pairing}) and the realized parameters tending to $1$, and is preserved by the surjection defining $T$.
\end{proof}

\section{Second operators and nonmaximal sums}
\label{app:partners}

We vary the partners of the constructed first operators and their affine transforms, using the pairing identity and missing primal points supplied by the construction theorem. The resulting families include arbitrarily small perturbations, strictly monotone partners, normal cones and convex penalties.

\subsection{Scalar second operators}
\label{subsec:counterexample-families}

Proposition~\ref{cor:positive-perturbations} gives nonmaximal sums for all positive rank-one coefficients. Now, we replace the linear scalar dependence by a nonlinear one. The first result uses the origin as its missing polar point, and the second allows that point to vary so that arbitrarily small bounded partners are included.

We construct nonlinear scalar partners $B_f$ for the same first operator $\mathcal A$ by imposing endpoint inequalities on $f$. The pairing identity then makes the origin a missing polar point of $\mathcal A+B_f$. The choices below include a smooth partner with bounded range.

\begin{proposition}[Scalar second operators]
\label{cor:scalar-partners}
Let $\mathcal A$ be as in Theorem~\ref{thm:endpoint-criterion}, and suppose Eq. (\ref{eq:abs-endpoint-exclusion}) holds. Let $f:\R\to\R$ be continuous and nondecreasing, with
\[
f(1)\ge S,\qquad f(-1)\le-S.
\]
Then $B_f x=f(\pair{x}{g})g$ is maximally monotone with full domain, and
\[
(0,0)\in\bigl(\gra(\mathcal A+B_f)\bigr)^\mu\setminus\gra(\mathcal A+B_f).
\]
In particular, this applies to
\[
\begin{aligned}
f(t)&=\kappa\tanh t,&\qquad \kappa&\ge\frac{S}{\tanh1},\\
f(t)&=\kappa|t|^{q-2}t,& q>1,\qquad \kappa&\ge S,
\end{aligned}
\]
where the second function is defined to be zero at $t=0$. The first choice gives a smooth nonlinear second operator with globally bounded range. Both choices are derivatives of finite convex functions composed with $x\mapsto\pair{x}{g}$.
\end{proposition}
\begin{proof}
The scalar monotonicity of $f$ makes $B_f$ continuous and monotone, so the direct polar argument in Proposition~\ref{cor:positive-perturbations}(2) proves its maximality. For $(x,a)=(La,a)\in\gra\mathcal A$, put $t=r(a)$ and $W=\|\omega(|t|-1)\|_{H_0}^2\le S$. Since $|t|>1$, the assumptions on $f$ give $tf(t)\ge S|t|$. Therefore,
\[
\pair{x}{a+B_fx}=t^2-1-W+tf(t)
\ge t^2-1+S(|t|-1)>0.
\]
The point $(0,0)$ is outside the sum graph because $\mathcal A(0)=\varnothing$, and $\dom B_f=X$ gives Eq. (\ref{eq:original-cq}). The displayed functions satisfy the scalar conditions. Finally, $B_f$ is the derivative of $x\mapsto\int_0^{\pair{x}{g}}f(t)\,dt$, which is convex.
\end{proof}

The scalar partners also admit arbitrarily small bounded perturbations of the same first operator, provided the missing polar point is allowed to vary.

\begin{proposition}[Small bounded scalar perturbations]
\label{cor:small-bounded-partners}
Let $\mathcal A,g$ be as in Proposition~\ref{cor:scalar-partners}. If $f:\R\to\R$ is continuous and nondecreasing with $f(1)>f(-1)$, then $B_fx=f(\pair{x}{g})g$ is maximally monotone with full domain and $\mathcal A+B_f$ is not maximally monotone. In particular, this holds for $f(t)=\kappa\tanh t$ for every $\kappa>0$, while $\sup_x\|B_fx\|=\kappa\|g\|\to0$ as $\kappa\downarrow0$.
\end{proposition}
\begin{proof}
Set $c=\frac{f(1)+f(-1)}2$, $\delta=\frac{f(1)-f(-1)}2>0$ and choose $0<s\le\frac{\sqrt\delta}{2}$. The actual labels from Assumption~\ref{ass:construction}\textup{(\ref{ass:realization})} give $p_s=\frac{a^{1+s}+a^{-1-s}}2$, with $r(p_s)=0$ and $Mp_s=(0,\omega(s))$. Thus $z_s=Lp_s\notin\dom\mathcal A$. For $a\in D$, put $t=r(a)$ and $u=|t|-1>0$. Monotonicity of $f$ gives $t(f(t)-c)\ge\delta|t|$, so
\[
\begin{aligned}
\pair{z_s-La}{p_s+cg-a-B_f(La)}
&=t^2-1-\|\omega(u)-\omega(s)\|_{H_0}^2+t(f(t)-c)\\
&\ge\delta-s^2+(2+2s+\delta)u>\frac{3\delta}{4}.
\end{aligned}
\]
Therefore $(z_s,p_s+cg)$ is a missing polar point. The full-domain maximality of $B_f$ follows from the continuous monotone argument in Proposition~\ref{cor:positive-perturbations}(2). The bound for $\tanh$ follows directly from its range.
\end{proof}

\subsection{Strictly monotone partners}
\label{subsec:appendix-strict}

Strict monotonicity of the second operator does not remove the failure. The next proposition replaces the rank-one partner by injective compact linear operators of infinite rank, and also gives smooth strictly monotone second operators with bounded range.

\begin{proposition}[Strictly monotone second operators]
\label{cor:strict-partners}
Let $A$ and $P$ be the operators on $c_0(I)$ in Example~\ref{thm:C}. Choose $w_i>0$ for $i\in I$ with $\sum_{i\in I}w_i<\infty$, and let $\varepsilon>0$.
\begin{enumerate}
\item For every $\lambda>0$, the operator
\[
B_{\lambda,\varepsilon}x=\lambda Px+\varepsilon(w_ix_i)_{i\in I}
\]
is bounded, linear, compact, injective and strictly monotone, with infinite rank. It is the derivative of the strictly convex quadratic
\[
\Phi_{\lambda,\varepsilon}(x)
=\frac{\lambda}{2}\pair{x}{g}^2
+\frac{\varepsilon}{2}\sum_{i\in I}w_ix_i^2.
\]
The operators $A$ and $B_{\lambda,\varepsilon}$ satisfy Eq. (\ref{eq:original-cq}) and have a nonmaximal sum. For fixed $\lambda$,
\[
\|B_{\lambda,\varepsilon}-\lambda P\|\le\varepsilon\sum_{i\in I}w_i\longrightarrow0
\quad(\varepsilon\downarrow0).
\]
\item For $\kappa\ge\frac{\sigma}{\tanh1}$, the operator
\[
\widetilde B_{\kappa,\varepsilon}x
=\kappa\tanh(\pair{x}{g})g
+\varepsilon(w_i\tanh x_i)_{i\in I}
\]
is smooth, globally Lipschitz, injective and strictly monotone, and its range has compact closure in $\ell^1(I)$. It is the derivative of the strictly convex function
\[
\Psi_{\kappa,\varepsilon}(x)
=\kappa\log\cosh(\pair{x}{g})
+\varepsilon\sum_{i\in I}w_i\log\cosh x_i.
\]
It is maximally monotone with full domain, and $A+\widetilde B_{\kappa,\varepsilon}$ is not maximally monotone.
\end{enumerate}
\end{proposition}
\begin{proof}
\noindent\textit{(1)} The map $Dx=(w_ix_i)_{i\in I}$ takes $c_0(I)$ boundedly into $\ell^1(I)$. Its finite-coordinate truncations converge in operator norm because the omitted weights have sum tending to zero. Thus $D$ and $B_{\lambda,\varepsilon}$ are compact. For $x\ne y$,
\[
\pair{x-y}{B_{\lambda,\varepsilon}x-B_{\lambda,\varepsilon}y}
=\lambda\pair{x-y}{g}^2+\varepsilon\sum_{i\in I}w_i(x_i-y_i)^2>0.
\]
Hence the map is strictly monotone and injective, and injectivity implies infinite rank. The displayed quadratic has this map as its Fr\'{e}chet derivative and is strictly convex by the same positive quadratic form. Proposition~\ref{cor:positive-perturbations}, with $R=\varepsilon D$, supplies maximality, Eq. (\ref{eq:original-cq}) and nonmaximality of the sum. The norm estimate follows from $\|Dx\|_1\le(\sum_iw_i)\|x\|_\infty$.

\noindent\textit{(2)} The map $Tx=(w_i\tanh x_i)_{i\in I}$ is monotone and globally Lipschitz. The strict increase of $\tanh$ and the positivity of every $w_i$ give
\[
\pair{x-y}{Tx-Ty}
=\sum_{i\in I}w_i(x_i-y_i)(\tanh x_i-\tanh y_i)>0
\quad(x\ne y).
\]
Consequently $\widetilde B_{\kappa,\varepsilon}$ is strictly monotone and injective. Its range satisfies
\[
\|\widetilde B_{\kappa,\varepsilon}x\|_1
\le\kappa\|g\|_1+\varepsilon\sum_iw_i.
\]
Moreover, the tails of the coordinate sums are uniformly bounded by the tails of $\kappa|g_i|+\varepsilon w_i$. Finite-coordinate truncation therefore proves relative compactness of the whole range. The bounds on the derivatives of $\tanh$, together with $\sum_iw_i<\infty$, justify differentiation of the series to every order. Thus the displayed potential is smooth, has derivative $\widetilde B_{\kappa,\varepsilon}$ and is strictly convex. The direct polar argument for continuous monotone maps gives maximality. Proposition~\ref{cor:scalar-partners}, with $S=\sigma$, gives $\pair{x}{a+\kappa\tanh(\pair{x}{g})g}>0$ on $\gra A$. Adding $\varepsilon\sum_iw_ix_i\tanh x_i\ge0$ preserves this inequality, so $(0,0)$ remains a point in the monotone polar outside the sum graph. Full domain gives Eq. (\ref{eq:original-cq}).
\end{proof}

We transfer the counterexample on $c_0(I)$ through a surjection from a space with a countable separating family of continuous linear functionals, then add a strictly monotone map on that space. This gives strictly monotone second operators with nonmaximal sums, including on standard $\ell^1$. The added map supplies strictness along the quotient kernel.

\begin{proposition}[Strict partners on the source space]
\label{cor:source-strict-partners}
Suppose $U$ has a countable family $(f_n)\subset U^*$ separating points, with $\|f_n\|\le1$, and $Q:U\to c_0(I)$ is a bounded linear surjection. Let $A,P$ be the operators in Example~\ref{thm:C}, and put $A_Q=Q^*AQ$, $P_Q=Q^*PQ$. Choose $w_n>0$ with $\sum_nw_n<\infty$ and define
\[
D_Uu=\sum_nw_nf_n(u)f_n,
\qquad
T_Uu=\sum_nw_n\tanh(f_n(u))f_n.
\]
Then the following are full-domain maximally monotone partners with which $A_Q$ has a nonmaximal sum:
\begin{enumerate}
\item $B_{\lambda,\varepsilon}=\lambda P_Q+\varepsilon D_U$, for every $\lambda,\varepsilon>0$, is bounded, linear, compact, injective and strictly monotone, with infinite rank.
\item $\widetilde B_{\kappa,\varepsilon}=Q^*B_\kappa Q+\varepsilon T_U$, where $B_\kappa x=\kappa\tanh(\pair{x}{g})g$ and $\kappa,\varepsilon>0$, is smooth, globally Lipschitz, injective and strictly monotone, and its range has compact closure.
\end{enumerate}
These conclusions apply to standard $\ell^1$, using its coordinate functionals and the specified quotient in Example~\ref{thm:L}.
\end{proposition}
\begin{proof}
The finite-rank truncations of $D_U$ converge in operator norm, since the tail norm is at most $\sum_{n>N}w_n$. Also
\[
\pair{h}{D_Uh}=\sum_nw_nf_n(h)^2>0\qquad(h\ne0).
\]
The series for $T_U$ converges uniformly, its finite sums have bounded finite-dimensional ranges, and its tails have norm at most $\sum_{n>N}w_n$. Hence its range has compact closure. Strict increase of $\tanh$ and separation by the $f_n$ give
\[
\pair{u-v}{T_Uu-T_Uv}>0\qquad(u\ne v).
\]
Bounded derivatives of $\tanh$ of every fixed order give uniformly convergent derivative series, so $T_U$ is smooth and globally Lipschitz. These maps are derivatives of the strictly convex functions $\frac12\sum_nw_nf_n(u)^2$ and $\sum_nw_n\log\cosh(f_n(u))$, respectively.

The stated properties now follow by adding the positive rank-one or bounded scalar pullback. Injectivity on the infinite-dimensional space $U$ also forces the linear partner to have infinite rank. Continuous monotonicity proves maximality. For either base partner, Proposition~\ref{cor:positive-perturbations} or Proposition~\ref{cor:small-bounded-partners} supplies a missing polar point $(z,p)$ with $z\notin\dom A$. Choose $Qu_0=z$. If $R=\varepsilon D_U$ or $\varepsilon T_U$, then
\[
\begin{aligned}
&\pair{u_0-u}{Q^*p+R(u_0)-Q^*a-Q^*B(Qu)-R(u)}\\
&\quad=\pair{z-Qu}{p-a-B(Qu)}+\pair{u_0-u}{R(u_0)-R(u)}\ge0.
\end{aligned}
\]
Since $u_0\notin\dom A_Q$, the shifted point is outside the sum graph. Both second operators have full domain.
\end{proof}

Every separable Banach space has such a countable separating family: take norm-one supporting functionals at a dense sequence of points of its unit sphere. The same perturbation proof applies to the closed-subspace examples in Proposition~\ref{cor:space-transfer}. It also applies to other spaces with a countable separating family, without requiring separability. On $\ell^1$, the two added maps have the simple formulas $(w_nu_n)_n$ and $(w_n\tanh u_n)_n$, taking values in $\ell^\infty$.

\subsection{Convex constraints and penalties}
\label{subsec:convex-partners}

The construction theorem also supplies second operators defined by convex constraints or convex penalties. First, we use its pairing identity for normal cones, then add finite continuous convex functions and use radial or seminorm potentials. These choices include both multivalued partners and a bounded single-valued partner on standard $\ell^1$.

% Bo\c{t} \cite{bot2026fpv} gives a halfspace normal-cone example for his geometric-series realization.
We use the pairing identity of the construction theorem to obtain a criterion for convex constraints, which we apply to closed balls. For a nonempty closed convex set $C\subset X$, define
\[
N_C(x)=
\begin{cases}
\{n\in X^*:\pair{y-x}{n}\le0\ \text{for all }y\in C\},&x\in C,\\
\varnothing,&x\notin C.
\end{cases}
\]
We give a condition on $C$ under which $(\mathcal A,N_C)$ is a counterexample pair. For the concrete operator on $c_0(I)$, an affine change yields such pairs with the normal cone of every closed ball of positive radius centered at zero.

\begin{proposition}[Normal cones and closed balls]
\label{cor:normal-cone}
Each of the following pairs consists of maximally monotone operators satisfying Eq. (\ref{eq:original-cq}) whose sum is not maximally monotone.
\begin{enumerate}
\item\label{cor:nc-general} The pair $(\mathcal A,N_C)$, where $\mathcal A$ is as in Theorem~\ref{thm:endpoint-criterion}, Eq. (\ref{eq:abs-endpoint-exclusion}) holds, and $C\subset X$ is closed and convex with
\begin{equation}
0\in C\subset\{x\in X:\pair{x}{g}\ge-1\},
\qquad
\dom\mathcal A\cap\operatorname{int}C\ne\varnothing.
\label{eq:cor-nc-condition}
\end{equation}
\item\label{cor:nc-ball} The pair $(A_r,B_r)$ for every $r>0$, where $A$ is the operator on $Y=c_0(I)$ from Example~\ref{thm:C}, $\overline B_Y=\{x\in Y:\|x\|_\infty\le1\}$, and
\begin{equation}
\begin{aligned}
c=\frac{15}{4}(e_{0,1}-e_{0,2}),~A_r(u)=\frac4r A\left(\frac4r u+c\right),~B_r=N_{r\overline B_Y}.
\end{aligned}
\label{eq:cor-ball-operators}
\end{equation}
\end{enumerate}
\end{proposition}
\begin{proof}
\noindent\textit{\textup{(\ref{cor:nc-general})}}
The operator $N_C$ is maximally monotone \cite{rockafellar1970subdiff}, with $\dom N_C=C$, so the interior-domain condition follows from Eq. (\ref{eq:cor-nc-condition}). For $x=La\in C$, $a\in D$, and $n\in N_C(x)$, set $t=r(a)=\pair{x}{g}$. Since $|t|>1$, $t\ge-1$ and $0\in C$, we have $t>1$ and $\pair{x}{n}\ge0$. Eq. (\ref{eq:abs-pairing}) yields
\[
\pair{-x}{-Sg-a-n}
\ge t^2-1-S+St
=(t-1)(t+1+S)>0.
\]
Thus $(0,-Sg)$ is monotonically related to $\gra(\mathcal A+N_C)$ and lies outside it because $\mathcal A(0)=\varnothing$.

\noindent\textit{\textup{(\ref{cor:nc-ball})}}
Set $a^\star=5e_{0,1}-7e_{0,2}+3e_{0,3}$. Since $\omega_1=0$, Eqs. (\ref{eq:D4})--(\ref{eq:D7}) give
\[
\begin{aligned}
r(a^\star)&=2,\\
ma^\star&=(1,0,0,\ldots)=F(2),
\end{aligned}
\]
and hence $a^\star\in D$, with $x^\star:=La^\star=e_{0,1}-e_{0,2}-3e_{0,3}\in\dom A$. For $C=c+4\overline B_Y$, we have
\[
\begin{aligned}
\|c\|_\infty&=\frac{15}{4}<4,\\
\|x^\star-c\|_\infty&=3<4,\\
\inf_{x\in C}\pair{x}{g}&=\pair{c}{g}-4\|g\|_1=-\frac12.
\end{aligned}
\]
These inequalities verify Eq. (\ref{eq:cor-nc-condition}), so \textup{(\ref{cor:nc-general})} applies to $(A,N_C)$ with $S=\sigma$.

For $\alpha=\frac4r$, the normal-cone definition gives $B_r(u)=\alpha N_C(\alpha u+c)$. The bijection $(x,a)\mapsto((x-c)/\alpha,\alpha a)$ preserves pairings of differences and transforms $(A,N_C)$ into $(A_r,B_r)$. It therefore preserves maximality and transforms the point $(0,-\sigma g)$ into
\begin{equation}
\begin{aligned}
(z_r,p_r)&=\left(-\frac{15r}{16}(e_{0,1}-e_{0,2}),-\frac{4\sigma}{r}g\right)\\
&\in(\gra(A_r+B_r))^\mu\setminus\gra(A_r+B_r).
\end{aligned}
\label{eq:cor-ball-witness}
\end{equation}
Finally, $(x^\star-c)/\alpha\in\dom A_r$ has norm $\frac{3r}{4}<r$, proving the interior-domain condition.
\end{proof}

Adding the subdifferential of a finite continuous convex function to either a positive rank-one partner or one of the preceding normal-cone partners preserves nonmaximality with the same first operator.

\begin{proposition}[Finite continuous convex penalties]
\label{cor:nonsmooth-penalties}
\label{cor:constrained-penalties}
Let $f:X\to\R$ be finite, continuous and convex, and let $\iota_C$ denote the function equal to zero on $C$ and $+\infty$ outside $C$.
\begin{enumerate}
\item For the operators in Proposition~\ref{cor:positive-perturbations} and every $\lambda>0$,
\[
B_f=\partial\left(\frac\lambda2\pair{\cdot}{g}^2+f\right)=\lambda\mathcal P+\partial f
\]
is maximally monotone with full domain, and $\mathcal A+B_f$ is not maximally monotone.
\item For the normal-cone examples in Proposition~\ref{cor:normal-cone}, the operator $\partial(\iota_C+f)=N_C+\partial f$ is maximally monotone with domain $C$ and gives a nonmaximal sum with the same first operator.
\end{enumerate}
\end{proposition}
\begin{proof}
A finite continuous convex function on a Banach space has a continuous linear subgradient at every point. The convex subdifferential sum rule, with the continuous summand $f$, gives the two equalities. Each resulting operator is the subdifferential of a proper lower semicontinuous convex function and hence is maximally monotone \cite{rockafellar1970subdiff}. Its domain is $X$ in part (1), and $C$ in part (2), since $\partial f(x)$ is nonempty and $0\in N_C(x)$ for every $x\in C$.

In part (1), choose an existing missing polar point $(z,p)$ of $\mathcal A+\lambda\mathcal P$ with $z\notin\dom\mathcal A$, and choose $q\in\partial f(z)$. For $a\in\mathcal A x$ and $b\in\partial f(x)$,
\[
\pair{z-x}{p+q-a-\lambda\mathcal Px-b}
=\pair{z-x}{p-a-\lambda\mathcal Px}+\pair{z-x}{q-b}\ge0.
\]
The shifted point remains outside the graph because its primal coordinate is unchanged. The same calculation proves part (2), using its existing witness $(z,p)$ with $z\in C\setminus\dom\mathcal A$. The domain condition is unchanged in both cases.
\end{proof}

Next, we construct full-domain second operators as subdifferentials of convex functions of the norm. A lower bound on the radial derivative makes the origin a missing polar point of the sum, and norm powers greater than one give coercive partners.

\begin{proposition}[Radial convex functions]
\label{cor:radial-partners}
Let $\mathcal A$ be as in Theorem~\ref{thm:endpoint-criterion}, and suppose Eq. (\ref{eq:abs-endpoint-exclusion}) holds. Put $d=\frac1{\|g\|}$. Let $\psi:[0,\infty)\to\R$ be continuously differentiable, convex and nondecreasing, and suppose
\[
\psi'(d)\ge\frac{S}{d}.
\]
Then $B_\psi=\partial(\psi(\|\cdot\|))$ is maximally monotone with full domain, and $\mathcal A+B_\psi$ is not maximally monotone under Eq. (\ref{eq:original-cq}). In particular, this holds for
\begin{equation}
B_{p,\kappa}=\partial\left(\frac{\kappa}{p}\|\cdot\|^p\right),
\qquad p\ge1,\qquad \kappa\ge S\|g\|^p.
\label{eq:cor-power-partners}
\end{equation}
For $p>1$, every $b\in B_{p,\kappa}(x)$ satisfies
\[
\frac{\pair{x}{b}}{\|x\|}=\kappa\|x\|^{p-1}\longrightarrow+\infty
\quad\text{as }\|x\|\to\infty.
\]
\end{proposition}
\begin{proof}
The function $x\mapsto\psi(\|x\|)$ is finite, continuous and convex, so its subdifferential is maximally monotone \cite{rockafellar1970subdiff}. At $x\ne0$, the Hahn--Banach theorem supplies $j\in X^*$ with $\|j\|=1$ and $\pair{x}{j}=\|x\|$. The supporting inequalities for $\psi$ and the norm show that $\psi'(\|x\|)j\in B_\psi(x)$. At zero, $0\in B_\psi(0)$ because $\psi$ is nondecreasing. Thus $\dom B_\psi=X$.

For any $b\in B_\psi(x)$ with $x\ne0$, testing its subgradient inequality at $(1+t)x$ for both signs of $t$ and letting $t\to0$ yields
\[
\pair{x}{b}=\|x\|\psi'(\|x\|).
\]
If $(x,a)\in\gra\mathcal A$, then $|\pair{x}{g}|>1$, so $\|x\|>d$. Eq. (\ref{eq:abs-pairing}) gives $\pair{x}{a}>-S$. Since $\psi'$ is nondecreasing,
\[
\pair{x}{a+b}>-S+\|x\|\psi'(\|x\|)
\ge-S+\|x\|\frac{S}{d}>0.
\]
Consequently $(0,0)$ belongs to the monotone polar of the sum and not to its graph. Full domain proves Eq. (\ref{eq:original-cq}). Substituting $\psi(r)=\frac{\kappa}{p}r^p$ gives Eq. (\ref{eq:cor-power-partners}) and the stated coercivity formula.
\end{proof}

In Proposition~\ref{cor:nonsmooth-penalties}, one may take $f(x)=\rho\|x\|$ or the maximum of finitely many affine functions. These give nonsmooth, generally multivalued second operators. We also construct partners from powers of a continuous seminorm that controls $|\pair{x}{g}|$, which supplies the lower pairing bound needed for a nonmaximal sum with the same first operator.

\begin{proposition}[Seminorm penalties]
\label{cor:seminorm-partners}
Let $\mathcal A,g,S$ be as in Theorem~\ref{thm:endpoint-criterion}, with Eq. (\ref{eq:abs-endpoint-exclusion}) satisfied. Let $q$ be a continuous seminorm with $|\pair{x}{g}|\le c q(x)$ for some $c>0$. For $p\ge1$ and $\kappa\ge Sc^p$, the operator
\[
B=\partial\left(\frac\kappa p q(\cdot)^p\right)
\]
is maximally monotone with full domain, and $(0,0)$ is a missing polar point of $\mathcal A+B$.
\end{proposition}
\begin{proof}
The potential is finite continuous convex, so its subdifferential is full-domain and maximally monotone. Testing the subgradient inequality for $b\in Bx$ at $(1+t)x$ and letting $t\to0$ from both sides gives $\pair{x}{b}=\kappa q(x)^p$. On $\gra\mathcal A$, $q(x)>1/c$ and $\pair{x}{a}>-S$. Hence $\pair{x}{a+b}>-S+\kappa q(x)^p>0$. The first operator has no value at zero, and full domain of the second gives the original qualification.
\end{proof}

This includes norms of finitely many linear observations that control $g$, as well as weighted anisotropic norms. In particular, $q(x)=|\pair{x}{g}|$ and $p=1$ give $B(x)=\kappa\operatorname{Sign}(\pair{x}{g})g$, where $\operatorname{Sign}(t)=\{\operatorname{sgn}t\}$ for $t\ne0$ and $\operatorname{Sign}(0)=[-1,1]$. A degenerate seminorm does not give coercivity in the ambient norm.

On standard $\ell^1$, we pair a rescaling of $T$ from Example~\ref{thm:L} with a fixed coordinatewise operator $C_*$. Its convex potential gives a single-valued, strictly monotone partner with bounded range whose sum with the rescaled operator is nonmaximal.

\begin{proposition}[A fixed coordinatewise partner on $\ell^1$]
\label{cor:fixed-coordinate-partner}
Let $T$ be the operator in Example~\ref{thm:L}, put $\widetilde T(x)=\frac12T(x/2)$, and index the coordinates of $\ell^1$ by $n\ge0$. Define
\[
w_n=2^{-n-2},\qquad (C_*x)_n=\frac{2x_n}{\sqrt{x_n^2+w_n^2}}.
\]
Then $C_*:\ell^1\to\ell^\infty$ is single-valued, strictly monotone and maximally monotone, with full domain and $\|C_*x\|_\infty\le2$. The pair $(\widetilde T,C_*)$ satisfies the interior-domain condition and has a nonmaximal sum.
\end{proposition}
\begin{proof}
The potential $\Psi(x)=2\sum_n(\sqrt{x_n^2+w_n^2}-w_n)$ is finite, convex and $2$-Lipschitz. Summing its coordinatewise supporting inequalities gives $C_*x\in\partial\Psi(x)$. Testing any subgradient along each positive and negative coordinate direction gives the reverse inclusion $\partial\Psi(x)=\{C_*x\}$. Subdifferential maximality therefore proves maximality, and strict increase of every coordinate function proves strict monotonicity.

The isomorphism $x\mapsto \frac{x}{2}$ in Theorem~\ref{lem:surjection-transport} gives maximality of $\widetilde T$. The bounds in Example~\ref{thm:L} yield $\|x\|_1>1$ and $\pair{x}{a}>-\sigma$ on its graph. Since $\sum_nw_n=\frac12$ and $r^2/\sqrt{r^2+w^2}>|r|-w$ for $w>0$,
\[
\pair{x}{C_*x}>2\|x\|_1-1>1\qquad(x\in\dom\widetilde T).
\]
Thus $\pair{x}{a+C_*x}>1-\sigma>0$, while $\widetilde T(0)=\varnothing$. This proves the missing origin witness, and $C_*$ has full domain.
\end{proof}

Unlike the smooth partners above, $C_*$ is not norm-continuous as a map from $\ell^1$ to $\ell^\infty$: for $x^{(n)}=w_ne_n\to0$, $\|C_*x^{(n)}\|_\infty=\sqrt2$. Its convex potential proves maximality.

For a maximally monotone operator whose domain excludes zero, we use a finite radial bound to construct a full-domain norm-subdifferential partner with bounded range and a nonmaximal sum. This extends the choice $p=1$ in Proposition~\ref{cor:radial-partners} beyond the construction theorem.

\begin{remark}[A second operator with bounded range]
\label{rem:bounded-range-partner}
Let $\mathcal T:X\rightrightarrows X^*$ be maximally monotone with nonempty domain, $0\notin\dom\mathcal T$ and $\mathcal V(\gra\mathcal T)\le C<\infty$. For $\rho\ge C$, define
\[
J_\rho(x)=
\{j\in X^*:\|j\|\le\rho,\ \pair{x}{j}=\rho\|x\|\}.
\]
This is the convex subdifferential of $x\mapsto\rho\|x\|$, so it is maximally monotone \cite{rockafellar1970subdiff}. Let $\overline B_{X^*}=\{j\in X^*:\|j\|\le1\}$. The Hahn--Banach theorem gives $\dom J_\rho=X$, and $J_\rho(0)=\rho\overline B_{X^*}$ gives $\operatorname{ran}J_\rho=\rho\overline B_{X^*}$. For every $(x,a)\in\gra\mathcal T$ and $j\in J_\rho(x)$,
\[
\pair{x}{a+j}\ge(\rho-C)\|x\|\ge0.
\]
Therefore, $(0,0)$ is monotonically related to $\gra(\mathcal T+J_\rho)$ and lies outside it. The sum is not maximally monotone, and the interior-domain condition holds because $J_\rho$ has full domain. For the first operators in Examples~\ref{thm:C} and~\ref{thm:L}, one may take any $\rho\ge2\sigma$.
\end{remark}

\section{Further consequences of pullback}
\label{app:pullback}

The pullback theorem and the extension correspondence yield constrained examples, prescribed affine factors and consequences for specific classes of Banach spaces. We also compute a dual quotient norm for the specified surjection onto $c_0$.

\subsection{Convex constraints and affine factors}
\label{subsec:appendix-affine-transfer}

First, we transfer counterexamples with normal-cone or convex-subdifferential partners through a bounded linear surjection. The subdifferential identity and the corresponding interior identity preserve the form of the second operator and the interior-domain condition. Then, we adjoin independent coordinates to prescribe affine factors.

\begin{corollary}[Convex constraints under surjective pullback]
\label{cor:pullback-constraints}
Let $Q:U\to X$ be a bounded linear surjection of real Banach spaces, and let $f:X\to\R\cup\{+\infty\}$ be proper, lower semicontinuous and convex. Then
\begin{equation}
\partial(f\circ Q)(u)=Q^*\partial f(Qu).
\label{eq:pullback-subgradient}
\end{equation}
For a nonempty closed convex set $C\subset X$, this gives
\[
N_{Q^{-1}C}=Q^*N_CQ,
\qquad
\operatorname{int}(Q^{-1}C)=Q^{-1}(\operatorname{int}C).
\]
Consequently, the closed-ball counterexamples yield second operators $N_{\{u:\|Qu\|\le r\}}$, and the radial counterexamples yield second operators $\partial(\frac\kappa p\|Q\cdot\|^p)$, with the corresponding pulled-back first operators.
\end{corollary}
\begin{proof}
When $f(Qu)$ is finite, testing a subgradient of $f\circ Q$ along $u+tk$, $k\in\ker Q$, shows that it annihilates $\ker Q$. Write it as $Q^*a$. Testing lifts of every $x\in X$ gives exactly the subgradient inequality for $a\in\partial f(Qu)$. The converse follows by composition. If $f(Qu)=+\infty$, both sides are empty. Taking $f=\iota_C$ proves the normal-cone identity. Continuity gives one inclusion for interiors, and openness of the surjection gives the other. Theorem~\ref{lem:surjection-transport} and the adjoint pairing identity now transfer maximality, the domain condition and the missing polar points of the stated examples.
\end{proof}

For example, on $U=X\oplus Z$ choose a bounded linear map $L_0:Z\to X$ and put $Q(x,z)=x+L_0z$. The resulting first operator has values $\{(a,L_0^*a):a\in A(x+L_0z)\}$, and the ball constraint becomes $\|x+L_0z\|\le r$. If $\ker Q\ne\{0\}$, this constraint is an unbounded cylinder and $\|Q\cdot\|$ is a seminorm. Pure pullback therefore does not preserve coercivity or strict monotonicity along the kernel.

Quotient maps can also prescribe entire affine fibres, and adjoining independent maximal graphs gives further counterexample families. The finite products below are equipped with the sum norm.

\begin{proposition}[Prescribed affine fibres and auxiliary factors]
\label{cor:product-fibres}
Let $(A,B)$ be one of the constructed counterexample pairs on $X$ with $\dom B=X$.
\begin{enumerate}
\item For arbitrary real Banach spaces $K_0,Z$, define on $X\oplus K_0\oplus Z$
\[
\widehat A(x,k,z)=
\begin{cases}A(x)\times\{0\}\times Z^*,&z=0,\\\varnothing,&z\ne0,\end{cases}
\qquad
\widehat B(x,k,z)=B(x)\times\{0\}\times\{0\}.
\]
This is a counterexample pair. A base missing polar point $(x_0,p_0)$ gives every missing point $((x_0,k,0),(p_0,0,z^*))$, $k\in K_0$, $z^*\in Z^*$.
\item Let $C,D:Z\rightrightarrows Z^*$ be maximally monotone with $\dom C\cap\operatorname{int}\dom D\ne\varnothing$. Then $(A\times C,B\times D)$ is a counterexample pair on $X\oplus Z$. In particular, $C$ can be any maximally monotone operator and $D=0$.
\end{enumerate}
\end{proposition}
\begin{proof}
\noindent\textit{(1)} Pull $A$ back to $X\oplus K_0$ by projection, then extend its graph to the closed subspace $X\oplus K_0\oplus\{0\}$ by allowing every dual extension, as in Proposition~\ref{cor:space-transfer}. Equivalently, freely varying the last dual coordinate forces the last primal coordinate of any polar point to vanish, after which maximality of the pulled-back graph gives membership. The second operator is the pullback of $B$ by projection and has full domain. Pairing each displayed witness against a point of the sum graph gives the corresponding base pairing, and membership would imply $p_0\in(A+B)x_0$.

\noindent\textit{(2)} For any maximally monotone graph $M$ and any point $(z,p)$, there is a graph point with pairing difference at most zero: use $(z,p)$ itself when it belongs to $M$, and a strictly negative test otherwise. A point outside the product of two maximal graphs therefore has a strictly negative test in one factor and a nonpositive test in the other. This proves maximality of both product operators. The interior of a finite product is the product of its interiors, so the domain condition holds. Choose $z\in\dom C\cap\dom D$, $c\in Cz$ and $d\in Dz$. If $(x_0,p_0)$ is a base missing polar point, then $((x_0,z),(p_0,c+d))$ is monotonically related to the product sum, by adding the two monotone inequalities, and cannot belong to it because its first component would give $p_0\in(A+B)x_0$.
\end{proof}

The first part prescribes independent primal and dual affine factors. The second part allows a prescribed monotone operator in the added coordinates. Taking $C=N_{D_0}$ for a nonempty closed convex set $D_0$ and taking the second auxiliary operator to be zero also prescribes a closed convex transverse domain. Finite iteration is allowed by the same proof.

\subsection{Changes of coordinates and classes of spaces}
\label{subsec:appendix-space-classes}

We show that the specified quotient $\ell^1\to c_0$ has no bounded linear right inverse. We also show that counterexample pairs persist under equivalent renorming, affine changes of coordinates, positive scaling and bounded skew additions, and give concrete spaces covered by Proposition~\ref{cor:space-transfer}.

\begin{remark}[Nonsplit quotients and changes of coordinates]
\label{rem:pullback-geometric-instances}
The surjection $Q:\ell^1\to c_0$ used in Example~\ref{thm:L} has no bounded linear right inverse. Indeed, such an inverse would embed $c_0$ into $\ell^1$. To see the obstruction directly, suppose $T:c_0\to\ell^1$ is bounded below and put $y_n=Te_n$, so $\|y_n\|_1\ge\delta>0$. Each coordinate of $y_n$ tends to zero. Select a subsequence and disjoint successive finite coordinate blocks so that each selected vector has total mass below $\delta/4$ outside its block. A single bounded sequence agreeing with its signs on these blocks then pairs with every selected vector by at least $\delta/2$. This contradicts the fact that its composition with $T$ belongs to $(c_0)^*=\ell^1$. Thus $\ker Q$ is uncomplemented, since a complement would give a bounded right inverse. By contrast, a surjection onto $\ell^1$ always has a bounded linear right inverse: choose uniformly bounded lifts $u_n$ of $e_n$ and use the absolutely convergent series $\sum a_nu_n$.

Equivalent norms leave these counterexample graphs unchanged, since they preserve the continuous dual and norm interiors. More generally, for a Banach-space isomorphism $L_0:U\to X$, $h\in X$, $a_0,b_0\in U^*$ and $c>0$, the pair
\[
\widetilde A(u)=cL_0^*A(L_0u+h)+a_0,
\qquad
\widetilde B(u)=cL_0^*B(L_0u+h)+b_0
\]
is again a counterexample. Pairings of differences are multiplied by $c$, and $(z,p)$ is transported to $(L_0^{-1}(z-h),cL_0^*p+a_0+b_0)$.

There is also a symmetry changing either first or second operators: if $S_A,S_B:X\to X^*$ are bounded linear maps with $\pair{x}{S_Ax}=\pair{x}{S_Bx}=0$, then $(A+S_A,B+S_B)$ is a counterexample with the same domains. The graph map $(x,a)\mapsto(x,a+Sx)$ preserves every pairing of differences and sends $(z,p)$ to $(z,p+Sz)$. Apply it to the two factors and to their sum. Taking $S_B=-S_A$ leaves the sum itself unchanged.
\end{remark}

We apply the preceding change of coordinates to prescribe the rank-one direction, which supplies the fixed-direction assertion in Corollary~\ref{cor:class-reflexivity}.

\begin{remark}[Prescribing the rank-one direction]
\label{rem:prescribed-direction}
On every space covered by Proposition~\ref{cor:space-transfer}, the nonzero direction of the second operator can be prescribed in advance. More precisely, for every $g\in X^*\setminus\{0\}$ there is one maximally monotone $A_g$ such that $A_g+\lambda P_g$ is nonmaximal for every $\lambda>0$, where $P_gx=\pair{x}{g}g$.

Choose $A_0,g_0$ supplied by that proposition. Since two proper linear hyperplanes cannot cover $X$, there is $e\in X$ with $g_0(e)g(e)\ne0$. Define the bounded linear isomorphism
\[
S_0x=x+\frac{g(x)-g_0(x)}{g_0(e)}e,\qquad
S_0^{-1}x=x-\frac{g(x)-g_0(x)}{g(e)}e.
\]
Direct substitution gives $S_0^*g_0=g$. With $A_g=S_0^*A_0S_0$, Theorem~\ref{lem:surjection-transport} gives maximality and transfers nonmaximality through
\[
A_g+\lambda P_g=S_0^*(A_0+\lambda P_{g_0})S_0.
\]
The choices of $S_0$ and $A_g$ are independent of $\lambda$.
\end{remark}

For concrete ambient spaces, Proposition~\ref{cor:space-transfer} gives such pairs on $\ell^\infty$, $C[0,1]$, $L^1[0,1]$, the real compact-operator space $\mathcal K(\ell^2_{\R})$, and the real trace-class space $\mathcal S_1(\ell^2_{\R})$. The first contains $c_0$ directly. Disjoint continuous functions of norm one embed $c_0$ into $C[0,1]$ by uniformly convergent sums. Normalized indicators of disjoint sets of positive measure embed $\ell^1$ isometrically into $L^1[0,1]$. Diagonal operators embed $c_0$ into $\mathcal K(\ell^2_{\R})$ in the operator norm and $\ell^1$ into $\mathcal S_1(\ell^2_{\R})$ in the trace norm. Each conclusion concerns maximality in the ambient space and its full continuous dual.

\subsection{Dual quotient norms and a fixed detecting vector}
\label{subsec:appendix-dual-quotient}

For the operator $T$ on standard $\ell^1$, we construct graph points with fixed primal norm and pairing whose dual quotient norms in $\ell^\infty/c_0$ tend to infinity. A single vector detects this divergence through pairings tending to $-\infty$, showing that the first two quantities do not control the dual values.

\begin{corollary}[Unbounded dual quotient norms at fixed primal norm and pairing]
\label{cor:dual-quotient-recovered}
For the operator $T$ in Example~\ref{thm:L}, there are $(u_N,p_N)\in\gra T$ and a fixed $d\in\ell^1$ such that
\[
\|u_N\|_1=8,\qquad \pair{u_N}{p_N}=3,\qquad
\dist_\infty(p_N,c_0)=\|p_N\|_\infty=5+2N,
\]
and, with $f=Q^*g$,
\[
\|d\|_1=\frac{\pi^2}{6},\qquad \pair{d}{f}=0,\qquad
\pair{d}{p_N}=-2\sum_{r\ge1}\frac{\min\{N,r\}}{r^2}\longrightarrow-\infty.
\]
The fibre of $(\gra T)^\mu$ over $d$ is empty. In particular, no finite-valued function of $\|u\|_1$ and $\pair{u}{p}$ bounds $\dist_\infty(p,c_0)$ on $\gra T$.
\end{corollary}
\begin{proof}
For the enumeration $(q_j)$ defining $Q$, every $a\in\ell^1(I)$ satisfies
\[
\dist_\infty(Q^*a,c_0)=\limsup_j|\pair{q_j}{a}|=\|a\|_1.
\]
The upper bound follows from $\|q_j\|_\infty\le1$. A finitely supported rational vector in the unit ball pairs arbitrarily close to $\|a\|_1$ in absolute value. Distinct rational scalar multiples approaching that vector supply infinitely many such tests, so deleting any finite prefix leaves the same supremum. Finally, the distance of a bounded sequence to $c_0$ is the limsup of the absolute coordinates, by truncation and the triangle inequality.

Take $a_N=-2e_{0,1}+3e_{0,3}+\sum_{b=1}^N(e_{b,1}-e_{b,2})$. Then $r(a_N)=2$, $ma_N=F(2)$ and
\[
x_N=La_N=-6e_{0,1}-8e_{0,2}-3e_{0,3}+\sum_{b=1}^N(e_{b,1}+e_{b,2}).
\]
Thus $\|x_N\|_\infty=8$, $\pair{x_N}{a_N}=3$ and $\|a_N\|_1=5+2N$. Choose $i_N$ with $q_{i_N}=x_N/8$ and set $u_N=8e_{i_N}$, $p_N=Q^*a_N$. These are graph points of $T$, and the adjoint identity gives all three asserted equalities.

For $v_r=\sum_{b=1}^r(e_{b,1}-e_{b,2})$, choose its index $j_r$ in the enumeration and set $d=-\sum_{r\ge1}r^{-2}e_{j_r}$. These indices are distinct, and $\pair{v_r}{g}=0$, $\pair{v_r}{a_N}=2\min\{N,r\}$. The asserted norm and pairings follow. The negative pairing has magnitude at least $2\sum_{r=1}^N1/r$, which tends to infinity. Since $\pair{d}{f}=0$ but every $u\in\dom T$ satisfies $|\pair{u}{f}|>1$, maximality of $T$ makes its polar fibre over $d$ empty. The final assertion follows by evaluating the proposed bound at the fixed arguments $(8,3)$.
\end{proof}

\section{Fitzpatrick functions, radial bounds and graph averages}
\label{app:fitzpatrick}

First, we prove the pullback formula for Fitzpatrick functions and their dual decompositions. Next, we show that the specified quotient preserves the exact radial bound in Corollary~\ref{cor:family-domain-geometry}. The remaining results distinguish radial bounds from Fitzpatrick values and domain membership, and give a finite averaging certificate for incompatible polar points of the sum.

\subsection{Fitzpatrick functions under surjective pullback}
\label{subsec:appendix-fitzpatrick-transfer}

Corollary~\ref{cor:fitzpatrick-values} computes the exact values for the constructed pair. The following formula transfers these values and the corresponding dual minimizers through every bounded linear surjection.

\begin{corollary}[Fitzpatrick functions under surjective pullback]
\label{cor:fitzpatrick-pullback}
Let $Q:U\to X$ be a bounded linear surjection of real Banach spaces. For an operator $T:X\rightrightarrows X^*$ with nonempty graph, put $T_Q=Q^*TQ$. Then
\begin{equation}
\mathsf F_{T_Q}(u,u^*)=
\begin{cases}
\mathsf F_T(Qu,p),&u^*=Q^*p,\ p\in X^*,\\
+\infty,&u^*\notin\operatorname{ran}Q^*.
\end{cases}
\label{eq:cor-fitzpatrick-pullback}
\end{equation}
For two operators $A,B$ with nonempty graphs, define
\[
\mathcal I_{A,B}(x,p)=\inf_{a\in X^*}\{\mathsf F_A(x,p-a)+\mathsf F_B(x,a)\}.
\]
Then
\[
\mathcal I_{A_Q,B_Q}(u,Q^*p)=\mathcal I_{A,B}(Qu,p).
\]
If this infimum is finite, its minimizers correspond through the injective map $a\mapsto Q^*a$. In particular, for the operators in Corollary~\ref{cor:fitzpatrick-values}, every $u\in\ker Q$ and $\lambda>0$ satisfy
\[
\mathsf F_{\mathcal A_Q+\lambda\mathcal P_Q}(u,0)=S-\lambda,
\qquad
\min_{u^*\in U^*}\{\mathsf F_{\mathcal A_Q}(u,-u^*)+\mathsf F_{\lambda\mathcal P_Q}(u,u^*)\}=S,
\]
with the minimum attained only at $u^*=0$.
\end{corollary}
\begin{proof}
Every primal point of $\gra T_Q$ can be translated by any element of $\ker Q$. Therefore, the supremum defining $\mathsf F_{T_Q}(u,u^*)$ is infinite unless $u^*$ annihilates $\ker Q$. Surjectivity and the open mapping theorem give $(\ker Q)^\perp=\operatorname{ran}Q^*$. If $u^*=Q^*p$, the adjoint identity and surjectivity reduce the defining supremum exactly to $\mathsf F_T(Qu,p)$, proving Eq. (\ref{eq:cor-fitzpatrick-pullback}).

In the dual decomposition of $Q^*p$, any summand outside $\operatorname{ran}Q^*$ gives an infinite value. The remaining summands are exactly $Q^*a$, with $a\in X^*$ unique, so Eq. (\ref{eq:cor-fitzpatrick-pullback}) proves the infimum identity and the assertion about finite minimizers. Apply the identity to $A=\mathcal A$, $B=\lambda\mathcal P$ and $Qu=0$. The equality $\mathcal A_Q+\lambda\mathcal P_Q=Q^*(\mathcal A+\lambda\mathcal P)Q$ proves the sum formula as well.
\end{proof}

\subsection{Exact radial bounds under the specified quotient}
\label{subsec:appendix-exact-radial}

The exact constant in Corollary~\ref{cor:family-domain-geometry} survives the specified quotient onto $c_0(I)$. This determines the precise strength of a norm-subdifferential partner for which the origin is a missing polar point.

\begin{remark}[Exact radial bounds and the norm partner]
\label{rem:sharp-radial-transfer}
For the quotient $Q$ in Eq. (\ref{eq:D10}),
\begin{equation}
\inf_{Qu=x}\|u\|_1=\|x\|_\infty\qquad(x\in c_0(I)).
\label{eq:metric-quotient-norm}
\end{equation}
Indeed, replace $\frac12$ in the residual construction of Lemma~\ref{lem:specified-quotient} by any $0<\delta<1$. Density of the same dictionary gives an actual preimage with $\|u\|_1\le\frac{\|x\|_\infty}{1-\delta}$. The contraction bound gives the reverse inequality, and letting $\delta\downarrow0$ proves Eq. (\ref{eq:metric-quotient-norm}).

Let $A=A_{w,B,\omega}$ and $T=Q^*AQ$. The pairing identity and Eq. (\ref{eq:metric-quotient-norm}) give
\[
\mathcal V(\gra T)=\mathcal V(\gra A)=2S.
\]
For the upper bound, use $\|Qu\|_\infty\le\|u\|_1$ on every lifted graph point. For the reverse bound, fix a graph point with negative pairing and choose preimages with norms approaching $\|x\|_\infty$ in Eq. (\ref{eq:metric-quotient-norm}), then take the supremum over these graph points. In particular, the two operators in Examples~\ref{thm:C} and~\ref{thm:L} have exact radial constant $2\sigma=\frac{\pi^2}{48}$.

For $M=A$ or $T$ and $J_\rho=\partial(\rho\|\cdot\|)$, $\rho\ge0$, every $j\in J_\rho(x)$ satisfies $\pair{x}{j}=\rho\|x\|$. Since $0\notin\dom M$, it follows that
\[
(0,0)\in\bigl(\gra(M+J_\rho)\bigr)^\mu\setminus\gra(M+J_\rho)
\quad\Longleftrightarrow\quad \rho\ge2S.
\]
This is a threshold for the specified origin witness, not a characterization of maximality of the sum when $\rho<2S$. Finally, the earlier domain and Fitzpatrick formulas give
\[
\mathsf F_A(0,0)=S,\qquad
\dist(0,\dom A)=\frac12,\qquad
\mathcal V(\gra A)=\frac{\mathsf F_A(0,0)}{\dist(0,\dom A)}.
\]
The last equality follows from the fixed-fibre approximation in Corollary~\ref{cor:family-domain-geometry}.
\end{remark}

\subsection{Prescribed radial bounds}
\label{subsec:appendix-radial}

Next, we add a scalar coordinate and translate its dual value to construct maximally monotone operators on standard $\ell^1$ with any prescribed positive radial bound and an infinite Fitzpatrick value at the origin. The same bound is the exact threshold for a norm-subdifferential partner to make the origin a missing polar point of the sum.

\begin{proposition}[Prescribed radial bound and infinite Fitzpatrick value]
\label{cor:radial-infinite-fitzpatrick}
For every $\theta>0$, there is a maximally monotone operator $T_\theta$ on standard $\ell^1$ such that
\[
0\notin\dom T_\theta,\qquad
\mathcal V(\gra T_\theta)=\theta,\qquad
\mathsf F_{T_\theta}(0,0)=+\infty.
\]
For $J_\rho=\partial(\rho\|\cdot\|_1)$, $\rho\ge0$, the point $(0,0)$ is monotonically related to $\gra(T_\theta+J_\rho)$ if and only if $\rho\ge\theta$. In that case it is outside the graph and the pair is a counterexample satisfying Eq. (\ref{eq:original-cq}).
\end{proposition}
\begin{proof}
Let $T$ be the operator in Example~\ref{thm:L}, put $C=2\sigma$ and $\alpha=\frac{\theta}{1+C}$, and first work on $E=\ell^1\oplus_1\R$. Define
\[
T_\theta(x,s)=\{(\alpha a,-\theta):a\in T(x)\}.
\]
Positive scaling of the dual variable preserves maximality. Pulling $\alpha T$ back by the projection $E\to\ell^1$ and translating its dual values by $(0,-\theta)$ therefore proves maximality of $T_\theta$. Its domain excludes zero. Since $\pair{x}{a}\ge-C\|x\|_1$,
\[
\pair{(x,s)}{(\alpha a,-\theta)}
\ge-\alpha C\|x\|_1-\theta|s|
\ge-\theta(\|x\|_1+|s|).
\]
Thus $\mathcal V(\gra T_\theta)\le\theta$. For a fixed $(x,a)\in\gra T$, letting $s\to+\infty$ gives
\[
\frac{\max\{0,-\alpha\pair{x}{a}+\theta s\}}{\|x\|_1+s}\longrightarrow\theta,
\qquad
-\alpha\pair{x}{a}+\theta s\longrightarrow+\infty.
\]
These limits give the reverse radial bound and the infinite Fitzpatrick value. For every $j\in J_\rho(u)$, $\pair{u}{j}=\rho\|u\|$. Hence the origin polar condition is exactly the uniform lower bound with constant $\rho$, which holds if and only if $\rho\ge\mathcal V(\gra T_\theta)=\theta$. The second factor is full-domain and maximally monotone. Finally, inserting the scalar coordinate identifies $E$ isometrically with standard $\ell^1$ and preserves all these assertions by Theorem~\ref{lem:surjection-transport}.
\end{proof}

For the concrete operators in Examples~\ref{thm:C} and~\ref{thm:L}, finite radial control also coexists with an empty value at the origin, separating this control from distance to the operator value.

\begin{remark}[The radial bound and distance to an empty value]
\label{rem:radial-distance-recovered}
For a maximally monotone operator $M$ with $0\notin\dom M$, define
\[
L_M(z,p)=\max\left\{0,\sup_{\substack{(x,a)\in\gra M\\x\ne z}}\frac{\pair{x-z}{p-a}}{\|x-z\|}\right\}.
\]
Then $L_M(0,0)=\mathcal V(\gra M)$. By Remark~\ref{rem:sharp-radial-transfer}, the two concrete operators satisfy $L_M(0,0)=2\sigma<\infty$ although $M(0)=\varnothing$ and $\dist(0,M(0))=+\infty$. Thus finite radial control does not imply domain membership or equality with the distance to the operator value. More precisely, for $J_\rho=\partial(\rho\|\cdot\|)$, $\rho>0$,
\[
(0,0)\in\bigl(\gra(M+J_\rho)\bigr)^\mu
\quad\Longleftrightarrow\quad\mathcal V(\gra M)\le\rho,
\]
because $\pair{x}{j}=\rho\|x\|$ for every $j\in J_\rho(x)$. For the concrete operators, $\rho\ge2\sigma$ is exactly the condition for this origin witness and supplies a full-domain, bounded-range partner with a nonmaximal sum.
\end{remark}

\subsection{Finite graph averages and incompatible polar points}
\label{subsec:appendix-graph-averages}

For the existing sum $\mathcal A+\mathcal P$, we certify nonuniqueness of maximal extensions using the average of two points of $\gra\mathcal A$ and a point of $\gra\mathcal P$. The resulting incompatible polar points and their pairing bounds transfer under surjective pullback.

\begin{corollary}[An averaging certificate for incompatible polar points]
\label{cor:averaging-recovered}
Let $\mathcal A,\mathcal P$ be as in Proposition~\ref{cor:positive-perturbations}, and put $G_+=\gra(\mathcal A+\mathcal P)$. There are the mean $q$ of two points of $\gra\mathcal A$, a number $0<\theta<1$ and $\varepsilon>0$ such that, with $q'=\theta q$ and $q_0=(0,0)$,
\[
q_0\in G_+^\mu,\qquad [q',v]\ge\varepsilon\quad(v\in G_+),\qquad
[q_0,q']=-\varepsilon,
\]
where $[(x,a),(y,b)]=\pair{x-y}{a-b}$. These certificates are preserved by the surjective pullback in Theorem~\ref{lem:surjection-transport}.
\end{corollary}
\begin{proof}
Choose $s>0$ small enough that $W=\|\omega(s)\|_{H_0}^2>E=(1+s)^2-1$, using $W\to S>0$. Choose actual labels $a^+,a^-$ at parameters $1+s,-1-s$, and put $v=(a^++a^-)/2$, $w=Lv$, $q=(w,v)$. The pairing identity gives
\[
r(v)=0,\qquad Mv=(0,\omega(s)),\qquad
\pair{w}{v}=-W,\qquad
\pair{La^\pm}{a^\pm}=E-W.
\]
In particular $(w,0)\in\gra\mathcal P$. Averaging the two monotonicity inequalities for $\mathcal A$ gives $\pair{w-x}{v-a}\ge-E$ for every $(x,a)\in\gra\mathcal A$. Adding the inequality for $\mathcal P$ at $(w,0)$ yields $[q,z]\ge-E$ for $z\in G_+$. The construction theorem gives $q_0\in G_+^\mu$, and $[q_0,q]=-W$. Set
\[
\theta=\frac{W-E}{2W},\qquad \varepsilon=\frac{(W-E)^2}{4W}.
\]
For each $z\in G_+$, the identity
\[
[\theta q,z]=(1-\theta)[q_0,z]+\theta[q,z]-\theta(1-\theta)[q_0,q]
\]
gives $[\theta q,z]\ge\theta((1-\theta)W-E)=\varepsilon$, whereas $[q_0,\theta q]=-\theta^2W=-\varepsilon$. Thus the two polar points cannot belong to the same monotone extension. Lifting the two actual primal points and applying the adjoint to their dual coordinates preserves the average, every pairing and the zero rank-one value at the averaged primal point.
\end{proof}

\begingroup
\raggedright \bibliographystyle{IEEEtran}
\bibliography{refer}
\endgroup

\end{document}